\documentclass[10pt]{amsart}
\usepackage{latexsym}
\usepackage{amsfonts}
\usepackage{amssymb}
\usepackage{amsmath}
\usepackage{mathrsfs}
\usepackage{wasysym}
\usepackage{hyperref}
\usepackage{tikz}
\usetikzlibrary{matrix,arrows}

\title[Connes' calculus for The Quantum double suspension]{Connes' calculus for the Quantum double suspension}
\author{Partha Sarathi Chakraborty}
\address{The Institute of Mathematical Sciences, CIT Campus, Taramani, Chennai
600113}
\email{parthac@imsc.res.in}

\author{Satyajit Guin}
\address{The Institute of Mathematical Sciences, CIT Campus, Taramani, Chennai
600113}
\email{gsatyajit@imsc.res.in}
\date{\today}

\newtheorem{definition}{Definition}[section]
\newtheorem{theorem}[definition]{Theorem}
\newtheorem{lemma}[definition]{Lemma}
\newtheorem{proposition}[definition]{Proposition}
\newtheorem{corollaryl}[definition]{Corollary}

\newtheorem{remark}[definition]{Remark}

\begin{document}
\begin{abstract}
Given  a spectral triple $(\mathcal{A},\mathcal{H},D)\,$ Connes associated a canonical differential graded algebra $\,\Omega_D^\bullet(\mathcal{A})$. However, so far this has been computed for 
very few special cases. We identify suitable hypotheses on a spectral triple that helps one to compute the associated Connes' calculus for its quantum double suspension. This allows one to compute $\,\Omega_D^\bullet$ for spectral triples obtained
by iterated quatum double suspension of the spectral triple associated with a first order differential operator on a  compact smooth manifold. This gives the first systematic computation of Connes' calculus for a large family of
spectral triples.
\end{abstract}
\maketitle
%\sffamily

%\vglue 1cm

\section{Introduction}
Noncommutative geometry of Connes is the study of interplay between an algebra $\mathcal A$ and a selfadjoint operator $D$, often referred as the Dirac operator represented on the same Hilbert space 
$\mathcal H$. In the last chapter of his book (\cite{4}), using these ingredients Connes also constructed a calculus $\,\Omega_D^\bullet(\mathcal{A})$. Please recall that by a calculus or a 
differential calculus one often means a differential graded algebra. It is shown in (\cite{4}) that using this calculus one can produce Hochschild
cocycle and cyclic cocycle (under certain assumptions) for Poincar\'{e} dual algebras which makes $\Omega_D^\bullet$ worth studying. Last but not the least Connes showed that in the case of 
classical spectral triple
associated to a compact Riemannian spin manifold, $\Omega_D^\bullet$ gives back the space of de-Rham forms on the manifold. This establishes $\,\Omega_D^\bullet$ as a genuine noncommutative
generalization of the classical de-Rham complex of a manifold. There are other instances of calculus as well, see for example (\cite{wor},\cite{pod},\cite{bma}).
For a better understanding of the Connes calculus it is imperative that we compute this in some cases.
 However outside the works of Connes there are very few instances (\cite{CSi},\cite{2}) where this calculus have been computed and have been put to investigation of concepts like Yang-Mills 
(\cite{CG2}). In view of this scenario, here in this article  we set ourselves with the task of computation of the Connes calculus for a certain systematic class of examples, which is by far missing till date.

The concept of quantum double suspension(QDS) of an algebra $\mathcal{A}$, denoted by $\varSigma^2 \mathcal{A}\, $, was introduced by Hong-Szymanski in (\cite{hsz}). Later quantum double suspension of a spectral triple was
introduced by Chakraborty-Sundar (\cite{3}) and a class of examples of regular spectral triple having simple dimension spectrum were constructed, useful in the context of local index formula of Connes-Moscovici (\cite{cmo}).
Note that iterating QDS on a manifold one can produce genuine noncommutative spectral triples. Under the following hypotheses

$\bullet\,\,[D,a]F-F[D,a]$ is a compact operator for all $a\in\mathcal{A}$, where $F$ is the sign of the operator $D$,

$\bullet\,\,\mathcal{H}^\infty:=\bigcap_{k\geq 1}\mathcal{D}om(D^k)$ is a left $\mathcal{A}$-module and $[D,\mathcal{A}]\subseteq\mathcal{A}\otimes\mathcal{E}nd_{\mathcal{A}}(\mathcal{H}^\infty)\subseteq
\mathcal{E}nd_{\mathbb{C}}(\mathcal{H}^\infty)$,\\
we compute $\,\Omega_D^\bullet$ for the quantum double suspended spectral triple $(\varSigma^2 \mathcal{A},\varSigma^2 \mathcal{H},\varSigma^2 D)$. 
Notable features of these hypotheses are, firstly the spectral triple associated with a first order elliptic differential operator on a manifold will
always satisfy them and secondly they are stable under quantum double suspension. Thus our results allows one to compute the Connes calculus  for spectral triples obtained  by
iteratively quantum double suspending spectral triples associated with first order differential operators on smooth compact manifolds. In particular iterated application of our construction 
on the spectral 
triple $\left(C^\infty(\mathbb{T}),L^2(\mathbb{T}),\frac{1}{i} \frac{d}{d\theta}\right)$  imply the computation of the Connes calculus for odd dimensional quantum spheres.
This  extends earlier work of (\cite{2}).

Organization of this paper is as follows. In section ($2$) we go through the definition of Connes' calculus $\,\Omega_D^\bullet$ and the quantum double suspension. Section ($3$) is devoted to the computation
of the space of forms $\,\Omega_{\varSigma^2 D}^\bullet(\varSigma^2\mathcal{A})\,$. In the last section we study compatible connections, curvatures on a {\it Hermitian} finitely generated projective 
(right)module (\cite{4}, Ch. $6$) over $\varSigma^2\mathcal{A}$. If we take $\mathcal{E}$ to be a Hermitian finitely generated projective module over $\mathcal{A}$, and denote the affine space of compatible connections by
$Con(\mathcal{E})$, then there is a canonical Hermitian finitely generated projective module over $\varSigma^2\mathcal{A}$ which we denote by $\widetilde{\mathcal{E}}$. The affine space of compatible connections on
$\widetilde{\mathcal{E}}$ is denoted by $Con(\widetilde{\mathcal{E}})$. We show that there is an affine embedding of $\,Con(\mathcal{E})$ into $\,Con(\widetilde{\mathcal{E}})$ which preserves the Grassmannian connection and
together with $Hom_\mathcal{A}\left(\mathcal{E},\mathcal{E}\otimes_\mathcal{A}\Omega_D^2(\mathcal{A})\right)$, the vector space containing the subspace of curvatures, these fit into a commutative diagram.

\medskip

\section{Preliminaries on Connes' Calculus and The Quantum Double Suspension}

In this section we recall the definition of Connes' calculus $\,\Omega_D^\bullet\,$ from (\cite{4}) and that of the quantum double suspension from (\cite{hsz}), (\cite{3}).
\begin{definition}
A {\it spectral triple} $(\mathcal{A},\mathcal{H},D)$, over an algebra $\mathcal{A}$ with involution $\star\,$ consists of the following things $:$
\begin{enumerate}
\item a $\,\star\,$-representation of $\mathcal{A}$ on a Hilbert space $\mathcal{H}$,
\item an unbounded selfadjoint operator $D$,
\item $D$ has compact resolvent and $[D,a]$ extends to a bounded operator on $\mathcal{H}$ for every $\,a \in \mathcal{A}$.
\end{enumerate}
\end{definition}

We shall assume that $\mathcal{A}$ is unital and the unit $1 \in \mathcal{A}$ acts as the identity on $\mathcal{H}$. If $|D|^{-p}$ is in the ideal of Dixmier traceable operators $\mathcal{L}^{(1,\infty)}$ then we say 
that the {\it spectral triple} is $\,p$-summable. Moreover, if there is a $\mathbb{Z}_2$-grading $\gamma\in\mathcal{B}(\mathcal{H})$ such that $\gamma$ commutes with every element of $\mathcal{A}$ and anticommutes with $D$
then the spectral triple $(\mathcal{A},\mathcal{H},D,\gamma)$ is said to be an {\it even spectral triple}.

\begin{definition}Let $\,\Omega^\bullet(\mathcal{A}) = \displaystyle\bigoplus_{k=0}^\infty \Omega^k(\mathcal{A})\,$ be the reduced universal differential graded algebra over $\mathcal{A}\, $. Here 
$\Omega^k(\mathcal{A}):=\mathcal{A}\otimes {\bar{\mathcal{A}}}^k\, $, $\, \bar{\mathcal{A}}=\mathcal{A}/\mathbb{C}\, $. The graded product is given by
\begin{eqnarray*}
&    & \left(\sum_k a_{0k}\otimes \overline{a_{1k}}\otimes \ldots \otimes \overline{a_{mk}}\right) . \left(\sum_{k^\prime}b_{0k^\prime}\otimes \overline{b_{1k^\prime}}\otimes \ldots \otimes \overline{b_{nk^\prime}}\right)\\
& := & \sum_{k,k^\prime} a_{0k}\otimes (\otimes_{j=1}^{m-1}\overline{a_{jk}})\otimes \overline{a_{mk}b_{0k^\prime}}\otimes (\otimes_{i=1}^n\overline{b_{ik^\prime}})
\end{eqnarray*}
\begin{eqnarray*}
&    & \quad+ \sum_{i=1}^{m-1}(-1)^i a_{0k}\otimes \overline{a_{1k}}\otimes \ldots \otimes \overline{a_{m-i,k}a_{m-i+1,k}}\otimes\ldots \otimes \overline{a_{mk}}\otimes (\otimes_{i=0}^n\overline{b_{ik^\prime}})\\
&    & \quad+ (-1)^ma_{0k}a_{1k}\otimes (\otimes_{j=2}^m\overline{a_{jk}})\otimes (\otimes_{i=0}^n\overline{b_{ik^\prime}})\, .
\end{eqnarray*}
for $\,\sum_k a_{0k}\otimes \overline{a_{1k}}\otimes \ldots \otimes \overline{a_{mk}}\in \Omega^m(\mathcal{A})$ and $\,\sum_{k^\prime}b_{0k^\prime}\otimes \overline{b_{1k^\prime}}\otimes \ldots \otimes \overline{b_{nk^\prime}}
\in \Omega^n(\mathcal{A})$. There is a differential $\,d\,$ acting on $\Omega^\bullet(\mathcal{A})$ which is given by
\begin{center}
$d(a_0\otimes \overline{a_1}\otimes \ldots \otimes \overline{a_k}) := 1\otimes \overline{a_0}\otimes \overline{a_1}\otimes \ldots \otimes \overline{a_k}\, \, \, \forall\, a_j \in \mathcal{A}\, ,$
\end{center}
and it satisfies the relations
\begin{enumerate}
 \item $d^2 \omega = 0\,, \, \, \forall\, \omega \in \Omega^\bullet(\mathcal{A})$,
 \item $d(\omega_1\omega_2) = (d\omega_1)\omega_2 + (-1)^{deg (\omega_1)} \omega_1 d\omega_2 \, ,\, \, \forall\, \omega_j \in \Omega^\bullet(\mathcal{A})$.
\end{enumerate}
We have a $*\,$-representation $\pi$ of $\,\Omega^\bullet(\mathcal{A})$ on $\mathcal{Q}(\mathcal{H}):=\mathcal{B}(\mathcal{H})/\mathcal{K}(\mathcal{H})$, given by
\begin{center}
$\pi(a_0\otimes \overline{a_1}\otimes \ldots \otimes \overline{a_k}) := a_0[D,a_1]\ldots [D,a_k]+\mathcal{K}(\mathcal{H})\, \, ;\, \, a_j \in \mathcal{A}\,$.
\end{center}
%Let $\,\theta:\mathcal{B}(\mathcal{H})\rightarrow\mathcal{Q}(\mathcal{H})$ be the projection map onto the Calkin algebra $\mathcal{Q}(\mathcal{H}):=\mathcal{B}(\mathcal{H})/\mathcal{K}(\mathcal{H})$. Then we get a
%representation $\,\theta\circ\pi:\Omega^\bullet(\mathcal{A})\longrightarrow\mathcal{Q}(\mathcal{H})$.
Let $J_0^{(k)} = \{\omega\in\Omega^k:\pi^k(\omega)=0\}$ and $J^\prime=\bigoplus J_0^{(k)}$. But $J^\prime$ fails to be a differential ideal in $\Omega^\bullet$. We consider $J^\bullet = \bigoplus J^{(k)}$ where
$J^{(k)}=J_0^{(k)}+dJ_0^{(k-1)}$. Then $J^\bullet$ becomes a differential graded two-sided ideal in $\Omega^\bullet$ and hence, the quotient $\Omega_D^\bullet=\Omega^\bullet/J^\bullet$ becomes a differential graded algebra,
called the Connes' calculus.
\end{definition}

The representation $\pi$ gives an isomorphism,
\begin{eqnarray}\label{main iso}
\Omega_D^k \cong \pi(\Omega^k)/\pi(dJ_0^{k-1})\,,\quad\forall\,k\geq 0 .
\end{eqnarray}
The differential $d$ on $\Omega^\bullet(\mathcal{A})$ induces a differential, denoted again by $d$, on the complex $\Omega^\bullet_D(\mathcal{A})$ so that we get a chain complex $(\,\Omega^\bullet_D(\mathcal{A}),d\,)$
and a chain map $\,\pi_D:\Omega^\bullet ( \mathcal{A} ) \rightarrow \Omega^\bullet_D(\mathcal{A})$ such that the following diagram 
\begin{center}
\begin{tikzpicture}[node distance=3cm,auto]
\node (Up)[label=above:$\pi_D$]{};
\node (A)[node distance=1.5cm,left of=Up]{$\Omega^\bullet(\mathcal{A})$};
\node (B)[node distance=1.5cm,right of=Up]{$\Omega^\bullet_D(\mathcal{A})$};
\node (Down)[node distance=1.5cm,below of=Up, label=below:$\pi_D$]{};
\node(C)[node distance=1.5cm,left of=Down]{$\Omega^{\bullet+1}(\mathcal{A})$};
\node(D)[node distance=1.5cm,right of=Down]{$\Omega^{\bullet+1}_D(\mathcal{A})$};
\draw[->](A) to (B);
\draw[->](C) to (D);
\draw[->](B)to node{{ $d$}}(D);
\draw[->](A)to node[swap]{{ $d$}}(C);
\end{tikzpicture} 
\end{center}
commutes.

\begin{remark}
In $($\cite{4}$)$ the definition of $\,\Omega_D^\bullet$ does not involve the projection map $\theta$ and Connes represented $\,\Omega^\bullet(\mathcal{A})$ on $\mathcal{B}(\mathcal{H})$ instead on $\mathcal{Q}(\mathcal{H})$.
However often, the explicit computation of $\,\Omega_D^\bullet$ is rather difficult, even in the particular cases. In $($\cite{2}$)$, authors have computed $\,\Omega_D^\bullet$ for the quantum $SU(2)$ and the Podle\'{s} sphere
by replacing $\mathcal{B}(\mathcal{H})$ with $\mathcal{Q}(\mathcal{H})$, i,e. following the above prescription. Justification for this has also been discussed in section $(3)$ of $($\cite{2}$)$.
\end{remark}

Now we define $\,\Omega_D^\bullet(\mathcal{A})$ for non-unital algebra $\mathcal{A}$. Notice that elements of $\,\Omega^k$ are linear combination of elements of the form $a_0da_1\ldots da_k$. For non-unital algebra
$\mathcal{A}$, one first considers the minimal unitization $\widetilde{\mathcal{A}}:=\mathcal{A}\oplus \mathbb{C}$ and embeds $\mathcal{A}$ in $\widetilde{\mathcal{A}}$ by the map $a\longmapsto (a,0)$. This makes
$\mathcal{A}$ an ideal in $\widetilde{\mathcal{A}}\,$. The map $\,(a,\lambda)\longmapsto \pi(a)+\lambda I\,$ gives a faithful representation of $\,\widetilde{\mathcal{A}}$ on $\mathcal{B}(\mathcal{H})$. Now, using the
embedding $\,\mathcal{A}\hookrightarrow\widetilde{\mathcal{A}}$, define elements of $\,\Omega^k(\mathcal{A})$ as linear combination of elements of the form $(a_0,0)d(a_1,0)\ldots d(a_k,0)$. Observe that $\,\Omega^n
(\mathcal{A})\subseteq\Omega^n(\widetilde{\mathcal{A}}\,)$ and hence $\frac{\Omega^n(\mathcal{A})}{\Omega^n(\mathcal{A})\bigcap J^n(\widetilde{\mathcal{A}}\,)}\subseteq\frac{\Omega^n(\widetilde{\mathcal{A}}\,)}
{J^n(\widetilde{\mathcal{A}}\,)}\,$ and subsequently we define $\,\Omega_D^\bullet(\mathcal{A})=\frac{\Omega^n(\mathcal{A})}{\Omega^n(\mathcal{A})\bigcap J^n(\widetilde{\mathcal{A}}\,)}$ for the nonunital case.

Now we define the concept of quantum double suspension introduced by Hong-Szymanski in (\cite{hsz}). Let $l$ denotes the left shift operator on $\ell^2(\mathbb{N})$ defined on the standard orthonormal basis $(e_n)$ 
by $l(e_n) = e_{n-1}$, $l(e_0) = 0$ and $N$ be the number operator on $\ell^2(\mathbb{N})$ defined as $N(e_n) = ne_n$. We fix the notation `$u$' throughout the article which denotes the rank one projection 
$|e_0\rangle \langle e_0| := I - l^*l$. Let $\mathcal{K}$ denotes space of compact operators on $\ell^2(\mathbb{N})$.

\begin{definition}
Let $\mathcal{A}$ be a unital $C^*$-algebra. The quantum double suspension of $\mathcal{A}$, denoted by $\varSigma^2 \mathcal{A}$, is the $C^*$-algebra generated by $a\otimes u$ and $1\otimes l$ in 
$\mathcal{A}\otimes \mathscr{T}$ where $\mathscr{T}$ is the Toeplitz algebra.
\end{definition}

There is a symbol map $\sigma : \mathscr{T} \longrightarrow C(S^1)$ which sends $l$ to the standard unitary generator $\,z$ of $C(S^1)$ and one has a short exact sequence
$$0 \longrightarrow \mathcal{K} \longrightarrow \mathscr{T} \stackrel{\sigma} \longrightarrow C(S^1) \longrightarrow 0$$
If $\rho$ denotes the restriction of $1\otimes \sigma$ to $\varSigma^2 \mathcal{A}$ then one gets the following short exact sequence
$$0 \longrightarrow \mathcal{A}\otimes \mathcal{K} \longrightarrow \varSigma^2 \mathcal{A} \stackrel{\rho} \longrightarrow C(S^1) \longrightarrow 0$$
There is a $\mathbb{C}\, $-linear splitting map $\sigma^\prime$ from $C(S^1)$ to $\varSigma^2 \mathcal{A}$ which sends $z$ to $1\otimes l$ and yields the following $\mathbb{C}\,$-vector spaces isomorphism $:$
\begin{center}
$\varSigma^2 \mathcal{A} \cong \left(\mathcal{A}\otimes \mathcal{K}\right) \bigoplus C(S^1)\,.$
\end{center}
Notice that $\sigma^\prime$ is injective since it has a left inverse $\, \rho\, $ and hence any $f \in C(S^1)$ can be identified with $1\otimes \sigma^\prime(f) \in \varSigma^2 \mathcal{A}$. For $f = \sum_n \lambda_n z^n \in 
C(S^1)\, $, we write $\, \sigma^{\prime}(f):= \sum_{n\geq 0} \lambda_nl^n + \sum_{n> 0} \lambda_{-n}l^{*n}$. The finite subalgebra, denoted by $(\varSigma^2 \mathcal{A})_{fin}$, is generated by $a\otimes T$ and 
$\sum_{0\leq n< \infty} \lambda_nl^n + \sum_{0< n< \infty} \lambda_{-n}l^{*n}$, where $a\in \mathcal{A}$ and $T\in \mathcal{B}\left(\ell^2(\mathbb{N})\right)$ is a finitely supported matrix.

\begin{definition}[\cite{3}]
For any spectral triple $(\mathcal{A},\mathcal{H},D)\,$,$\, (\varSigma^2 \mathcal{A},\varSigma^2 \mathcal{H}:=\mathcal{H}\otimes \ell^2(\mathbb{N}),\varSigma^2 D:= D\otimes I + F\otimes N)$ becomes a spectral triple 
where $F$ is the sign of the operator $D$ and $N$ is the number operator on $\ell^2(\mathbb{N})$. It is called the quantum double suspension of the spectral triple $(\mathcal{A},\mathcal{H},D)$.
\end{definition}

%It is easy to see that if $(\mathcal{A},\mathcal{H},D)$ is $p$-summable then $(\varSigma^2 \mathcal{A},\varSigma^2 \mathcal{H},\varSigma^2 D)$ is a $p+1$-summable spectral triple.
Notice that $\left(\mathbb{C}[z,z^{-1}],\ell^2(\mathbb{N}),N\right)$ is also a spectral triple, and for any $f \in \mathbb{C}[z,z^{-1}]\,$, we have $\,[\varSigma^2 D,1\otimes \sigma^\prime(f)] = F\otimes [N,f]$. Here
we record two conditions on a spectral triple which will be used later.
\medskip

\textbf{Conditions}~:\\
 $\hspace*{.7cm}(A)\,\,[D,a]F-F[D,a]$ is a compact operator for all $a\in\mathcal{A}\,$.\\
 $\hspace*{.7cm}(B)\,\,\mathcal{H}^\infty:=\bigcap_{k\geq 1}\mathcal{D}om(D^k)$ is a left $\mathcal{A}$-module and $\,[D,\mathcal{A}]\subseteq\mathcal{A}\otimes\mathcal{E}nd_{\mathcal{A}}(\mathcal{H}^\infty)\subseteq
\mathcal{E}nd_{\mathbb{C}}(\mathcal{H}^\infty)\,$.
\medskip

\begin{proposition}\label{justification for assumption}
These conditions are valid for the classical case where $\mathcal{A}=C^\infty(\mathbb{M})$ and $D$ is a first order differential operator. Moreover, if a spectral triple $(\mathcal{A},\mathcal{H},D)$ satisfies these
conditions then the quantum double suspended spectral triple $((\varSigma^2\mathcal{A})_{fin},\varSigma^2\mathcal{H},\varSigma^2D)$ also satisfies them.
\end{proposition}

\begin{proof}
When $D$ is of order $1,\,[D,a]F-F[D,a]$ has order $-1$ and hence it is a compact operator. Now suppose $[D,a]F-F[D,a]$ is a compact operator for all $a\in\mathcal{A}\,$. To check the stability under QDS,
note that $(\varSigma^2\mathcal{A})_{fin}=\mathcal{A}\otimes\mathcal{S}\bigoplus \mathbb{C}[z,z^{-1}]$ as linear space and $sign(\varSigma^2D)=F\otimes1$. Now
\begin{eqnarray*}
&   & [\varSigma^2D,a\otimes T+f](F\otimes 1)-(F\otimes 1)[\varSigma^2D,a\otimes T+f]\\
& = & [D,a]F\otimes T+FaF\otimes[N,T]+1\otimes[N,f]-F[D,a]\otimes T-a\otimes[N,T]-1\otimes[N,f]\\
& = & [D,a]F\otimes T-F[D,a]\otimes T+[F,a]F\otimes[N,T].
\end{eqnarray*}
This says that $[\varSigma^2D,(\varSigma^2\mathcal{A})_{fin}](F\otimes 1)-(F\otimes 1)[\varSigma^2D,(\varSigma^2\mathcal{A})_{fin}]$ is also a compact operator on $\mathcal{H}\otimes\ell^2(\mathbb{N})$. The second condition
follows similarly.
\end{proof}

\begin{lemma}\label{FA and A has no intersection}
Let $\,\sigma_1,\,\sigma_2$ denote any two $2\times 2$ Pauli spin matrices. For a given spectral triple $(\mathcal{A},\mathcal{H},D)$ consider the even spectral triple $(\widetilde{\mathcal{A}},\widetilde{\mathcal{H}},
\widetilde{D},\gamma)$ where $\,\widetilde{\mathcal{H}}=\mathcal{H}\otimes\mathbb{C}^2,\,\widetilde{\mathcal{A}}=\mathcal{A}\otimes I_2,\,\widetilde{D}=D\otimes\sigma_1,\,\gamma=1\otimes\sigma_2\,$. Then
$\,\Omega_{\widetilde{D}}^\bullet(\widetilde{\mathcal{A}})\cong\Omega_D^\bullet(\mathcal{A})$.
\end{lemma}

\begin{proof}
First observe that $\,\sum \widetilde{a_0}\prod_{i=1}^n [\widetilde D,\widetilde{a_i}]=(\sum a_0\prod_{i=1}^n [D,a_i])\otimes\sigma_1^n\,$ where $\widetilde{a_i}=a_i\otimes I_2$. Now
$$\sigma_1^n=\begin{cases}
               \sigma_1\quad for\,\,n\,\,odd\,,\\
               I_2\quad for\,\,n\,\,even\,.
              \end{cases}
$$ immediately shows that $\pi(\Omega^n(\widetilde{\mathcal{A}}))\cong\pi(\Omega^n(\mathcal{A}))$ for all $n\geq 1$. Since $\,\sum a_0\prod_{i=1}^n [D,a_i]\otimes\sigma_1^n=0$ implies $\prod_{i=0}^n [D,a_i]\in\pi(dJ_0^n
(\mathcal{A}))$, we have $\pi(dJ_0^n(\widetilde{\mathcal{A}}))\cong\pi(dJ_0^n(\mathcal{A}))$ for all $n\geq 1$. This completes the proof.
\end{proof}

\begin{remark}\label{essential remark}
$(a)~$ Proposition $(\,\ref{justification for assumption}\,)$ says that iterating the classical case of spectral triples canonically associated with a first order differential operator on a compact manifold, one gets a lot
of examples satifying our conditions.\\
%\begin{enumerate}
% \item $[D,a]F-F[D,a]$ is a compact operator for all $a\in\mathcal{A}\,$,
% \item $\,\mathcal{H}^\infty:=\bigcap_{k\geq 1}\mathcal{D}om(D^k)$ is a left $\mathcal{A}$-module and $\,[D,\mathcal{A}]\subseteq\mathcal{A}\otimes\mathcal{E}nd_{\mathcal{A}}(\mathcal{H}^\infty)\subseteq
%\mathcal{E}nd_{\mathbb{C}}(\mathcal{H}^\infty)\,$.
%\end{enumerate}
$(b)~$ Observe that $sign(\widetilde{D})=sign(D)\otimes\sigma_1$ and hence if $(\mathcal{A},\mathcal{H},D)$ satisfies our conditions then so does $(\widetilde{\mathcal{A}},\widetilde{\mathcal{H}},
\widetilde{D},\gamma)$. For any even spectral triple it is obvious that $\,F\mathcal{A}\cap\mathcal{A}=\{0\}$ where $F=sign(D)$. Hence, Lemma $(\,\ref{FA and A has no intersection}\,)$ will guarantee that in our context,
without loss of generality, we can always take $\,F\mathcal{A}\cap\mathcal{A}=\{0\}$. Throughout the article we stick to this fact.
\end{remark}
\medskip
                      
\textbf{Notation}~: \begin{enumerate}
                     \item In this article we will work with $(\varSigma^2 \mathcal{A})_{fin}$ and denote it by $\,\varSigma^2 \mathcal{A}\,$ for notational brevity.
                     \item For all $f\in\mathbb{C}[z,z^{-1}],$ we denote $\,[N,f]$ by $\,f^\prime\,$ for notational brevity.
                     \item `$\mathcal{S}$' denotes the space of finitely supported matrices in $\mathcal{B}(\ell^2(\mathbb{N}))\, $.
                     \item Any $T=(T_{ij})\in\mathcal{S}$ is said to has order $m$ if $m\geq 1$ is the least natural number such that $T_{ij}=0$ for all $i,j>m$.
                     \item $(e_{ij})$ will denote infinite matrix with $1$ at the $ij$-th place and zero elsewhere. We call it elementary matrix.
                    \end{enumerate}
\medskip

\section{Connes' Calculus for The Quantum Double Suspension}

In this section we take a spectral triple satisfying Conditions $(A),(B)$. Because of Remark ($\,$\ref{essential remark}$\,$) we can assume $F\mathcal{A}\cap\mathcal{A}=\{0\}$. Our goal here is computation of
$\,\Omega_{\varSigma^2D}^\bullet\left((\varSigma^2\mathcal{A})_{fin}\right)\, $. Note that $(\mathcal{S},\ell^2(\mathbb{N}),N)$ is a spectral triple but $\mathcal{S}$ is non-unital. We first consider
$\Omega_N^\bullet(\mathcal{S})$, following the definition of $\,\Omega_D^\bullet\,$ for non-unital algebras. We need to compute this complex first.
\begin{lemma}\label{finite support 1}
$\pi_N\left(\Omega^n(\mathcal{S})\right) = \mathcal{S}\, \, \, $ for all $\, \, n\geq 0$.
\end{lemma}

\begin{proof}
Let's take $n > 0$. We have to show $\pi_N\left(\Omega^n(\mathcal{S})\right) \supseteq \mathcal{S}$. Choose any $T\in \mathcal{S}$ of order $l$. If $n$ is even then take $n= 2r$. Now observe that,
\begin{eqnarray*}
T & = & \sum_{j=1}^{l-1} T\left(\left[N,\frac{1}{j-l}(e_{j,l})\right]\left[N,\frac{1}{l-j}(e_{l,j})\right] \right)^{r}
   + T\left(\left[N,\frac{1}{l-1}(e_{l,1})\right]\left[N,\frac{1}{1-l}(e_{1,l})\right] \right)^{r}.
\end{eqnarray*}
For $n=2r+1$ we can similarly write,
\begin{eqnarray*}
T & = & \sum_{j=1}^{l-1}\, T(e_{j,l})\left(\left[N,\frac{1}{l-2}(e_{l,2})\right]\left[N,\frac{1}{2-l}(e_{2,l})\right]\right)^{r}\left[N,\frac{1}{l-j}(e_{l,j})\right]\\
  &   & \quad\,\,+\, T(e_{l,1})\left[N,\frac{1}{1-l}(e_{1,l})\right] \left(\left[N,\frac{1}{l-1}(e_{l,1})\right]\left[N,\frac{1}{1-l}(e_{1,l})\right] \right)^r
\end{eqnarray*} and this completes the proof.
\end{proof}

\begin{lemma}\label{finite support 2}
$\pi_N\left(dJ_0^n(\mathcal{S})\right) = \mathcal{S}\, \, \, $ for all $\, \, n\geq 1$.
\end{lemma}

\begin{proof}
Notice that $\pi_N\left(dJ_0^n(M_l(\mathbb{C}))\right)$ is an ideal in $M_l(\mathbb{C})$ for any $l$ and hence if we can produce one 
nontrivial element then $\,\pi_N\left(dJ_0^n(M_l(\mathbb{C}))\right)$ will be equal with $M_l(\mathbb{C})$ for all $\,l$.
\newline
For $n=1\, $, choose
$$\zeta  =  (e_{2,3})d\left(\frac{1}{2}(e_{3,1})\right) - (e_{2,2})d\left((e_{2,1})\right)\,.$$
For $n=2\, $, choose
$$\zeta  =  \left((e_{2,3})d\left(\frac{1}{2}(e_{3,1})\right) - (e_{2,2})d((e_{2,1}))\right)d\left(\frac{-1}{3}(e_{1,4})\right)\,.$$
For $n=3\, $, choose
$$\zeta  =  \left((e_{2,3})d\left(\frac{1}{2}(e_{3,1})\right) - (e_{2,2})d((e_{2,1}))\right)d\left(\frac{-1}{3}(e_{1,4})\right)d\left(\frac{1}{2}(e_{4,2})\right)\,.$$
For $n \geq 4\, $, choose
\begin{eqnarray*}
\zeta & = & \left((e_{2,n+1})d\left(\frac{1}{n}(e_{n+1,1})\right) - (e_{2,2})d((e_{2,1}))\right)d\left(\frac{-1}{3}(e_{1,4})\right)\bullet\\
      &   &  \,\,\,\prod_{j=4}^n d((-1)(e_{j,j+1}))d(\frac{1}{n-1}(e_{n+1,2})).
\end{eqnarray*}
It is easy to see that for all $\,n\geq 1$ these elements lie in $J_0^n \left( M_l(\mathbb{C}) \right)$. One can verify that $\pi(d\zeta) \neq 0$ in each cases. 
\end{proof}

\begin{proposition}[]\label{complex for finite support}
For $(\mathcal{S},\ell^2(\mathbb{N}),N)$ we have
\begin{enumerate}
 \item  $\Omega^n_N(\mathcal{S}) = \mathcal{S}\, \, $ for $\, n=0\, ,1\, $.
 \item  $\Omega^n_N(\mathcal{S}) = 0\, \, $ for all $\, n\geq 2$.
\end{enumerate}
\end{proposition}

\begin{proof}
Combine Lemma ($\,$\ref{finite support 1}$\,$) and ($\,$\ref{finite support 2}$\,$). 
\end{proof}

Now we are ready for the computation of $\,\Omega_{\varSigma^2D}^\bullet\left(\varSigma^2\mathcal{A}\right)\, $. Note that both $\pi_{\varSigma^2D}\left(\Omega^\bullet(\mathcal{A}\otimes \mathcal{S})\right)$ and
$\pi_{\varSigma^2D}\left(\Omega^\bullet(\mathbb{C}[z,z^{-1}])\right)$ are subspaces of $\pi_{\varSigma^2D}\left(\Omega^\bullet(\varSigma^2\mathcal{A})\right)$, since $\varSigma^2\mathcal{A}=\mathcal{A}\otimes
\mathcal{S}\bigoplus \mathbb{C}[z,z^{-1}]$ as $\mathbb{C}$-vector spaces. Furthermore, $\pi_{\varSigma^2D}\left(\Omega^\bullet(\mathbb{C}[z,z^{-1}])\right)=F^\bullet\otimes \pi_N\left(\Omega^\bullet(\mathbb{C}[z,z^{-1}])
\right)$.

\begin{lemma}\label{referring lemma 0}
$\pi\left(\Omega^1(\varSigma^2 \mathcal{A})\right) = \pi\left(\Omega^1(\mathcal{A}\otimes \mathcal{S})\right) + (F\otimes 1)\left(\mathcal{A}\otimes \mathcal{S} + \pi(\Omega^1(\mathbb{C}[z,z^{-1}])) \right)$.
\end{lemma}

\begin{proof}
Note that $\pi\left(\Omega^1(\mathcal{A}\otimes \mathcal{S})\right) \subseteq \pi\left(\Omega^1(\varSigma^2 \mathcal{A})\right)$. Now for any element $(F\otimes 1)(a\otimes T + f_0f_1^\prime)$ of
$(F\otimes 1)\left(\mathcal{A}\otimes \mathcal{S} + \pi(\Omega^1(\mathbb{C}[z,z^{-1}])) \right)$ we see that,
\begin{eqnarray*}
 (F\otimes 1)(a\otimes T + f_0f_1^\prime) & = & (a\otimes T + f_0f_1^\prime)(F\otimes 1)\\
                              & = & (a\otimes T + f_0f_1^\prime)(1\otimes l)(1\otimes l^*)(F\otimes 1)\\
                              & = & (a\otimes T + f_0f_1^\prime)(1\otimes l)\left[\varSigma^2 D,1\otimes l^*\right]
\end{eqnarray*}
This is clearly in $\pi\left(\Omega^1(\varSigma^2 \mathcal{A})\right)$. To see the reverse inclusion we start with arbitrary element $\, \sum_k (a_{0k}\otimes T_{0k} + f_{0k})\left[\sum^2 D, a_{1k}\otimes T_{1k}+ f_{1k}
\right]$ of $\pi\left(\Omega^1(\varSigma^2 \mathcal{A})\right)$. Then,
\begin{eqnarray*}
&   & \sum_k \left(a_{0k}\otimes T_{0k} + f_{0k}\right)\left[\varSigma^2 D, a_{1k}\otimes T_{1k}+ f_{1k}\right]\\
&   & = \sum_k \left(a_{0k}\otimes T_{0k}\right) \left[\varSigma^2 D, a_{1k}\otimes T_{1k}\right] + \left(a_{0k}\otimes T_{0k} + f_{0k}\right) \left[\varSigma^2 D, f_{1k}\right]\\
&   & \, \, \, \, \, \, \, \, \, \, \, \, \, + \left[\varSigma^2 D,a_{1k}\otimes f_{0k}T_{1k}\right] - Fa_{1k}\otimes f_{0k}^\prime T_{1k}
\end{eqnarray*}
which is an element of $\pi\left(\Omega^1(\mathcal{A}\otimes \mathcal{S})\right) + (F\otimes 1)\left(\mathcal{A}\otimes \mathcal{S} + \pi(\Omega^1(\mathbb{C}[z,z^{-1}])) \right)$.
\end{proof}

\begin{lemma}\label{referring lemma 1}
$\pi \left(\Omega^n(\varSigma^2 \mathcal{A})\right) = \sum_{j=1}^n F^{j-1}\pi \left(\Omega^{n+1-j}(\mathcal{A}\otimes \mathcal{S})\right) + (F\otimes 1)^n\left(\mathcal{A}\otimes \mathcal{S}+\pi(\Omega^n(\mathbb{C}[z,z^{-1}])
)\right)$ for all $n\geq 1\,$.
\end{lemma}

\begin{proof}
We prove by induction. Suppose the statement is true for $n=k$. Any element of $\pi \left(\Omega^{k+1}(\varSigma^2 \mathcal{A})\right)$ can be written as $\,\omega [\varSigma^2 D,a\otimes T + f]$, where
$\omega$ is in $\pi \left(\Omega^k(\varSigma^2 \mathcal{A})\right)$. By assumtion $\,\omega = \sum_{i=1}^{k+1}\omega_i$ where $\omega_i \in F^{i-1}\pi\left(\Omega^{k+1-i}(\mathcal{A}\otimes \mathcal{S})\right)$ 
for $1\leqslant i \leqslant k\,$ and $\,\omega_{k+1} \in (F\otimes1)^k \left(\mathcal{A}\otimes \mathcal{S} + \pi(\Omega^k(\mathbb{C}[z,z^{-1}]))\right)$. Hence,
\begin{center}
$\omega [\varSigma^2 D,a\otimes T + f] = \sum_{i=1}^{k+1}\omega_i[\varSigma^2 D,a\otimes T] + \sum_{i=1}^{k+1}\omega_i(F\otimes f^\prime)$.
\end{center}
This is an element of $\, \sum_{j=1}^{k+1} F^{j-1}\pi \left(\Omega^{k+2-j}(\mathcal{A}\otimes \mathcal{S})\right) + (F\otimes 1)^{k+1}\left(\mathcal{A}\otimes \mathcal{S}+\pi\left(\Omega^{k+1}(\mathbb{C}[z,z^{-1}])\right)
\right)$. To get the reverse inclusion one can use the same trick used in Lemma ($\, $\ref{referring lemma 0}$\, $).
\end{proof}

\begin{lemma}\label{referring lemma 2}
$\pi \left(\Omega^1(\mathcal{A}\otimes \mathcal{S})\right)\, \, \bigcap\, \, (F\otimes 1)\left(\mathcal{A}\otimes \mathcal{S}+ \pi(\Omega^1(\mathbb{C}[z,z^{-1}]))\right) = (F\otimes 1)(\mathcal{A}\otimes \mathcal{S})$.
\end{lemma}

\begin{proof}
Choose any arbitrary element $\sum_k \left(a_{0k}\otimes T_{0k}\right) \left[\varSigma^2 D,a_{1k}\otimes T_{1k} \right]$ of $\pi \left(\Omega^1(\mathcal{A}\otimes \mathcal{S})\right)$. 
In terms of elementary matrices $((e_{ij}))$ we can rewrite this element as following~,
\begin{eqnarray*}
\sum_k \left(a_{0k}\otimes T_{0k}\right) \left[\varSigma^2 D,a_{1k}\otimes T_{1k} \right] & = & \sum_k \left(\sum_{i,j} a_{0kij}\otimes e_{ij}\right)\left[\varSigma^2 D, \sum_{p,q} a_{1kpq}\otimes e_{pq} \right]\\
                                                                                          & = &  \sum_{k,i,j,q} a_{0kij}\left([D,a_{1kjq}] + F(j-q)a_{1kjq}\right) \otimes e_{iq}
\end{eqnarray*}
Now any element of $(F\otimes 1)(\mathcal{A}\otimes S) + (F\otimes 1)\pi\left(\Omega^1(\mathbb{C}[z,z^{-1}])\right)$ looks like $\sum_{k^\prime} Fa_{k^\prime}\otimes T_{k^\prime} + F\otimes f$ for some $f \in\mathbb{C}[z,z^{-1}]$.
Equality of these two elements shows that $f$ has to be a compact operator on $\ell^2(\mathbb {N})$ (Take any linear functional $\int$ on $B(\mathcal{H})$ and hit both elements by 
$\int \otimes Id$). Hence, if intersection is nontrivial then it must be contained in the ideal $(F\otimes 1)(\mathcal{A}\otimes \mathcal{S})$. We now show that $(F\otimes 1)(\mathcal{A}\otimes \mathcal{S})$
is contained in the intersection. Choose any arbitrary element $\sum_{k^\prime} Fa_{k^\prime}\otimes T_{k^\prime}$. Consider the following equation~,
\begin{eqnarray*}
\sum_{k^\prime} Fa_{k^\prime}\otimes T_{k^\prime} & = & \sum_k \left(a_{0k}\otimes T_{0k}\right) \left[\varSigma^2 D,a_{1k}\otimes T_{1k} \right].
\end{eqnarray*}
Choose $a_{1k}=1$ for each $k$. Then this equation reduces to,
\begin{eqnarray*}
\sum_{k^\prime} Fa_{k^\prime}\otimes T_{k^\prime} & = & \sum_k Fa_{0k}\otimes T_{0k}[N,T_{1k}].
\end{eqnarray*}
Using Lemma ($\,$\ref{finite support 1}$\,$) we can write each $T_{k^\prime}$ as $\sum_{m< \infty} T_{0m}^{(k^\prime)}[N,T_{1m}^{(k^\prime)}]$.
Hence, this equation has nontrivial solution, which shows that $(F\otimes 1)(\mathcal{A}\otimes \mathcal{S})  \subseteq \pi \left(\Omega^1(\mathcal{A}\otimes \mathcal{S})\right)$.
\end{proof}

\begin{lemma}\label{referring lemma 3}
$(F\otimes 1)\pi\left(\Omega^n(\mathcal{A}\otimes \mathcal{S})\right) \subseteq \pi\left(\Omega^{n+1}(\mathcal{A}\otimes \mathcal{S})\right)\,$ for all $\,n \geq 1$.
\end{lemma}

\begin{proof}
This follows from the fact that for algebra $\mathcal{B}$ $$\Omega^n(\mathcal{B})=\underbrace{\Omega^1(\mathcal{B})\otimes_\mathcal{B}\ldots\otimes_\mathcal{B}\Omega^1(\mathcal{B})}_{n\,\, times}\,.$$ Now use Lemma
($\,$\ref{referring lemma 2}$\,$) which says that $(F\otimes 1)(\mathcal{A}\otimes \mathcal{S})\subseteq\pi \left(\Omega^1(\mathcal{A}\otimes \mathcal{S})\right)$.
\end{proof}

\begin{proposition}\label{first main proposition}
$\pi \left(\Omega^n(\varSigma^2 \mathcal{A})\right) = \pi \left(\Omega^n(\mathcal{A}\otimes \mathcal{S})\right)\, \bigoplus\, \, \pi\left(\Omega^n(\mathbb{C}[z,z^{-1}])\right)$ for all $n\geq 0\,$. 
\end{proposition}

\begin{proof}
Combine Lemmas ($\,$\ref{referring lemma 1}$\,$), ($\,$\ref{referring lemma 2}$\,$) and ($\,$\ref{referring lemma 3}$\,$).
\end{proof}

Recall that (the isomorphism in ($\,$\ref{main iso}$\,$)),  $\Omega_{\varSigma^2 D}^n\left(\varSigma^2 \mathcal{A}\right)\cong\pi\left(\Omega^n(\varSigma^2\mathcal{A})\right)/\pi\left(dJ_0^{n-1}(\varSigma^2\mathcal{A})\right)$. 
Hence our next target is to identify the quotient $\,\pi\left(dJ_0^{n-1}(\varSigma^2 \mathcal{A})\right)$.

\begin{lemma}
$\pi \left(dJ_0^1(\varSigma^2 \mathcal{A})\right) = \pi \left(dJ_0^1(\mathcal{A}\otimes \mathcal{S})\right) \bigoplus \, \pi \left(dJ_0^1(\mathbb{C}[z,z^{-1}])\right)$.
\end{lemma}

\begin{proof}
Suppose $\zeta = \sum_k d(a_{0k}\otimes T_{0k} + f_{0k})d(a_{1k}\otimes T_{1k} + f_{1k})$ be an element of $dJ_0^1(\varSigma^2 \mathcal{A})$. Then,
\begin{eqnarray*}
\sum_k (a_{0k}\otimes T_{0k} + f_{0k})[\varSigma^2 D,a_{1k}\otimes T_{1k} + f_{1k}] = 0
\end{eqnarray*}
Then $\,\pi(\zeta) = \sum_k [\varSigma^2 D,a_{0k}\otimes T_{0k} + f_{0k}][\varSigma^2 D,a_{1k}\otimes T_{1k} + f_{1k}]\,$ equals to the following sum,
\begin{eqnarray*}
&   & \sum_k\,\, [\varSigma^2 D,a_{0k}\otimes T_{0k}][\varSigma^2 D,a_{1k}\otimes T_{1k}] + [\varSigma^2 D,f_{0k}][\varSigma^2 D,f_{1k}]\\
&   & \quad\quad +\, [\varSigma^2 D,a_{0k}\otimes T_{0k}][\varSigma^2 D,f_{1k}] + [\varSigma^2 D,f_{0k}][\varSigma^2 D,a_{1k}\otimes T_{1k}].
\end{eqnarray*}
The term $\sum_k [\varSigma^2 D,f_{0k}][\varSigma^2 D,f_{1k}]$ lies in $\pi \left(dJ_0^1(\mathbb{C}[z,z^{-1}])\right)$. If we can write each term $[\varSigma^2 D,a_{0k}\otimes T_{0k}][\varSigma^2 D,f_{1k}]\, \, $ as
$\, \, \sum_{k^\prime} [\varSigma^2 D,b_{0k^\prime}\otimes S_{0k^\prime}][\varSigma^2 D,b_{1k^\prime}\otimes S_{1k^\prime}]\, \, $ such that 
\begin{eqnarray*}
(a_{0k}\otimes T_{0k})[\varSigma^2 D,f_{1k}] & = & \sum_{k^\prime} (b_{0k^\prime}\otimes S_{0k^\prime})([\varSigma^2 D,b_{1k^\prime}\otimes S_{1k^\prime}])
\end{eqnarray*}
where $S_{0k^\prime}$ and $S_{1k^\prime}$'s are from $\mathcal{S}$ and similarly for the term $[\varSigma^2 D,f_{0k}][\varSigma^2 D,a_{1k}\otimes T_{1k}]\, $, then we can conclude that 
$\pi \left(dJ_0^1(\varSigma^2 \mathcal{A})\right)\subseteq \pi \left(dJ_0^1(\mathcal{A}\otimes \mathcal{S})\right) \bigoplus \, \pi \left(dJ_0^1(\mathbb{C}[z,z^{-1}])\right)$. First let $b_{1k^\prime} = 1$ for all $k^\prime$.
Then we have the following equations to solve,
\begin{eqnarray}
Fa_{0k}\otimes T_{0k}f_{1k}^\prime & = & \sum_{k^\prime} Fb_{0k^\prime}\otimes S_{0k^\prime}[N,S_{1k^\prime}]
\end{eqnarray}
$$\quad\quad F[D,a_{0k}]\otimes T_{0k}f_{1k}^\prime+a_{0k}\otimes[N, T_{0k}]f_{1k}^\prime = \,\sum_{k^\prime}F[D,b_{0k^\prime}]\otimes S_{0k^\prime}[N,S_{1k^\prime}]+b_{0k^\prime}\otimes[N,S_{0k^\prime}][N,S_{1k^\prime}]\,.$$
For that it is enough to solve the following equations
\begin{eqnarray}\label{equation}
T_{0k}f_{1k}^\prime & = & \sum_{k^\prime} S_{0k^\prime}[N,S_{1k^\prime}]
\end{eqnarray}
$$\,\,\,[N, T_{0k}]f_{1k}^\prime = \,\,\sum_{k^\prime} [N,S_{0k^\prime}][N,S_{1k^\prime}]\,.$$
Note that $f_{1k}^\prime$ is of the form $\,\sum_{i=1}^n \lambda_i z^i + \sum_{j=1}^m \lambda_{-j} {(z^{-1})}^j$. Then $\sigma^\prime(f_{1k}^\prime) \in \mathcal{B}\left(\ell^2(\mathbb{N})\right)$ is the following matrix
$$
\begin{pmatrix}
 0 & \lambda_1 & \lambda_2 & \ldots & \ldots & \lambda_n & 0 & \ldots & \ldots \ldots\\
 \lambda_{-1} & 0 & \lambda_1 & \lambda_2 & \ldots & \ldots & \lambda_n & 0 & \ldots \ldots\\
 \ldots & \ldots & \ldots & \ldots & \ldots & \ldots & \ldots & \ldots & \ldots \ldots\\
 \lambda_{-m} & \lambda_{-m+1} & \ldots & \ldots & \ldots & \ldots & \ldots & \ldots & \ldots \ldots \\
 0 & \lambda_{-m} & \lambda_{-m+1} & \ldots & \ldots & \ldots & \ldots & \ldots & \ldots \ldots\\
 0 & 0 & \lambda_{-m} & \lambda_{-m+1} & \ldots & \ldots & \ldots & \ldots & \ldots \ldots\\
 \ldots & \ldots & \ldots & \ldots & \ldots & \ldots & \ldots & \ldots & \ldots \ldots
\end{pmatrix}.
$$
Each row and column of this matrix has only finitely many non-zero entries and all the diagonal entries are zero. Denote this matrix by $(\beta_{pq})_{p,q}$. Notice that $(e_{ij}).\sigma^\prime(f_{1k}^\prime)$ is the matrix
whose $i$-th row consists of $j$-th row of $(\beta_{pq})_{p,q}$ and zero everywhere else, whereas $\sigma^\prime(f_{1k}^\prime).(e_{ij})$ is the matrix whose $j$-th column consists of $i$-th column of $(\beta_{pq})_{p,q}$ and
zero everywhere else. Let $\,T_{0k}$ be denoted by the matrix $(\alpha_{ij})_{ij}$. Since $\,T_{0k}\in\mathcal{S}$, one can assume that $\alpha_{ij}=0$ for all $i\geq r+1$ and $j\geq s+1$, for some $r,s$.
Observe that $\,T_{0k}\sigma^\prime(f_{1k}^\prime)=(\alpha_{ij})_{ij}(\widetilde{\beta_{pq}})_{p,q}\,$, where $$\widetilde{\beta_{pq}}=\begin{cases}
                                                                                                                                        \begin{array}{lcl}
                                                                                                                                         \beta_{pq}\quad for\,\,\, 1\leq p\leq s\,,\,q\leq n+s\,,\\
                                                                                                                                         0\quad\quad otherwise\,.
                                                                                                                                        \end{array}
                                                                                                                                       \end{cases}
$$
Hence $(\widetilde{\beta_{pq}})_{p,q}\in\mathcal{S}$ and we have a solution for equation ($\,$\ref{equation}$\,$). Similarly one can do for the term $[\varSigma^2 D,f_{0k}][\varSigma^2 D,a_{1k}\otimes T_{1k}]$.
\end{proof}

\begin{proposition}\label{forms for denominator}
$\pi \left(dJ_0^n(\varSigma^2 \mathcal{A})\right) = \pi \left(dJ_0^n(\mathcal{A}\otimes \mathcal{S})\right) \bigoplus \, \pi \left(dJ_0^n(\mathbb{C}[z,z^{-1}])\right)$ for all $n\geq 1\,$.
\end{proposition}

\begin{proof}
For arbitrary `$n$' it follows from our observation in the previous Lemma that both $[\varSigma^2 D,a\otimes T][\varSigma^2 D,f]$ and $[\varSigma^2 D,g][\varSigma^2 D,a^\prime\otimes T^\prime]$ for $f,g\in\mathbb{C}[z,z^{-1}]$
can be replaced by $[\varSigma^2 D,b\otimes S][\varSigma^2 D,b^\prime\otimes S^\prime]\, $ where $\, T,T^\prime,S,S^\prime$ all lie in $\mathcal{S}$.
\end{proof}

\begin{lemma}\label{forms for circle 1}
For the spectral triple $\left(\mathbb{C}[z,z^{-1}],\ell^2(\mathbb{N}),N\right)$, we have $\,\pi_{N}\left(\Omega^n(\mathbb{C}[z,z^{-1}])\right) = \mathbb{C}[z,z^{-1}]\, $ for all $\, n \geq 0$.
\end{lemma}

\begin{proof}
Clearly $\pi\left(\Omega^n(\mathbb{C}[z,z^{-1}])\right) \subseteq \mathbb{C}[z,z^{-1}]\, $. For the other inclusion consider $\xi,\eta\in \mathbb{C}[z,z^{-1}]$ where $\xi=z,\eta=z^{-1}$. Then $[N,\xi]=\xi$ and
$[N,\eta]=-\eta$.
\newline
\underline{Case $\, 1\, :$} $\, \, $ Suppose $n = 2r$ is even. Choose any $\phi \in \mathbb{C}[z,z^{-1}]$ and consider $\omega = \phi \underbrace{(d\xi d\eta)\ldots (d\xi d\eta)}_r \in \Omega^n\left(\mathbb{C}[z,z^{-1}]\right)
$. Then $$\pi(\omega) = \phi(\,\underbrace{[N,\xi][N,\eta])\ldots ([N,\xi][N,\eta]}_r\,)\,.$$ But $\,[N,\xi][N,\eta]= -1$. This proves that $\mathbb{C}[z,z^{-1}] \subseteq \pi\left(\Omega^n(\mathbb{C}[z,z^{-1}])\right)$.
\newline
\underline{Case $\, 2\, :$} $\, \, $ Suppose $n=2r+1$ is odd. Choose any $\phi \in C(S^1)$. Consider $\omega = \phi \xi d\eta \underbrace{(d\xi d\eta)\ldots (d\xi d\eta)}_r \in \Omega^n\left(C(S^1)\right)$
for $r\neq 0$ and $\,\omega = \phi \xi d\eta$ for $r=0$.
\end{proof}

\begin{lemma}\label{forms for circle 2}
For the spectral triple $\left(\mathbb{C}[z,z^{-1}],\ell^2(\mathbb{N}),N\right)$, we have $\,\pi_N\left(dJ_0^n(\mathbb{C}[z,z^{-1}])\right) = \mathbb{C}[z,z^{-1}]\, $ for all $\, n \geq 1$.
\end{lemma}

\begin{proof}
It is clear from previous Lemma ($\, $\ref{forms for circle 1}$\, $) that $\pi\left(dJ_0^n(\mathbb{C}[z,z^{-1}])\right) \subseteq \mathbb{C}[z,z^{-1}]\, $. For the other inclusion notice that for
$n \geq 1,\, \, \pi\left(dJ_0^n(\mathbb{C}[z,z^{-1}])\right)\, $ is an ideal in $\mathbb{C}[z,z^{-1}]$. We will show that $1\in \mathbb{C}[z,z^{-1}]$ lies in $\pi\left(dJ_0^n(\mathbb{C}[z,z^{-1}])\right)\, $. Consider
$\,\xi=z,\,\eta=z^{-1}\in\mathbb{C}[z,z^{-1}]$.
\newline
\underline{Case $\, 1\, :$} $\, \, $ For $n \geq 3\, $ odd, consider
\begin{eqnarray*} 
\omega = \xi d\eta \underbrace{(d\xi d\eta)\ldots (d\xi d\eta)}_{(n-1)/2}\, +\, \eta d\xi \underbrace{(d\xi d\eta)\ldots (d\xi d\eta)}_{(n-1)/2} \, \, \in \Omega^n\left(\mathbb{C}[z,z^{-1}]\right). 
\end{eqnarray*}
\underline{Case $\, 2\, :$} $\, \, $ For $n\geq 2\, $ even, consider
\begin{eqnarray*} 
\omega = -\xi^2 d\eta d\eta\underbrace{(d\xi d\eta)\ldots (d\xi d\eta)}_{(n-2)/2}\, +\, \eta^2 d\xi d\xi\underbrace{(d\xi d\eta)\ldots (d\xi d\eta)}_{(n-2)/2} \, \, \in \Omega^n\left(\mathbb{C}[z,z^{-1}]\right). 
\end{eqnarray*}
\underline{Case $\, 3\, :$} $\, \, $ For $n=1\, $, consider $\omega = \xi d\eta\, +\, \eta d\xi$.
\newline
One can check that for all $n\geq 1$, $\pi(\omega)=0$ i,e. $\omega \in J_0^n\left(\mathbb{C}[z,z^{-1}]\right)$. But
\begin{center}
$\pi(d\omega) = \begin{cases}
                 \begin{array}{lcl}
                  -2\quad\quad\quad\quad\quad\quad\quad\, for\,\,\,\, n=1\,,\\
                  -2(-1)^{(n-1)/2}\quad\quad for\,\,\,\, n\geq 3\,\, odd\,,\\
                  -4(-1)^{(n-2)/2}\quad\quad for\,\,\,\, n\geq 2\,\, even\,.
                 \end{array}
                \end{cases}
$\end{center}
This justifies our claim.
\end{proof}

\begin{proposition}\label{forms for circle}
For the spectral triple $\left(\mathbb{C}[z,z^{-1}],\ell^2(\mathbb{N}),N\right)$,
\begin{enumerate}
 \item $\Omega_N^n\left(\mathbb{C}[z,z^{-1}]\right) = \mathbb{C}[z,z^{-1}]\, $,~ for $\, n = 0\, ,1\, $.
 \item $\Omega_N^n\left(\mathbb{C}[z,z^{-1}]\right) = 0\, $,~ for $\, n \geq 2\, $.
\end{enumerate}
\end{proposition}

\begin{proof}
Combine Lemma ($\,$\ref{forms for circle 1}$\,$) and ($\,$\ref{forms for circle 2}$\,$).
\end{proof}

\begin{proposition}\label{imp theorem}
For $\, (\varSigma^2 \mathcal{A},\mathcal{H}\otimes \ell^2(\mathbb{N}),\varSigma^2 D)\, $,
\begin{enumerate}
 \item $\Omega_{\varSigma^2 D}^1(\varSigma^2 \mathcal{A}) \cong \Omega_{\varSigma^2 D}^1(\mathcal{A}\otimes \mathcal{S}) \bigoplus \,\mathbb{C}[z,z^{-1}]\,$,
 \item $\Omega_{\varSigma^2 D}^n(\varSigma^2 \mathcal{A}) \cong \Omega_{\varSigma^2 D}^n(\mathcal{A}\otimes \mathcal{S})\,\,$, for all $\, \, n\geq 2$.
\end{enumerate}
\end{proposition}

\begin{proof}
Use Propositions ($\,$\ref{first main proposition}$\,$), ($\,$\ref{forms for denominator}$\,$) and ($\,$\ref{forms for circle}$\,$).
\end{proof}

Our next goal is to determine $\Omega_{\varSigma^2 D}^n\left(\mathcal{A}\otimes \mathcal{S}\right)$ in terms of $\Omega_D^n(\mathcal{A})$. Note that we are viewing $\mathcal{A}\otimes \mathcal{S}$ inside
the unital algebra $\varSigma^2 \mathcal{A}$ as an embedded subspace.

\begin{lemma}\label{n forms involving A,S}
$\pi\left(\Omega^n(\mathcal{A}\otimes \mathcal{S})\right) = \sum_{r=0}^n F^r\pi(\Omega^{n-r}(\mathcal{A}))\otimes \mathcal{S}\, $ for all $n\geq 0\,$.
\end{lemma}

\begin{proof}
The inclusion `$\subseteq$' is obvious since one just has to expand the commutators $[\varSigma^2 D,\, .\,]$ involved in LHS. For `$\supseteq$' we show that $F^r\pi\left(\Omega^{n-r}(\mathcal{A})\right)\otimes \mathcal{S}
\subseteq\pi\left(\Omega^n(\mathcal{A}\otimes \mathcal{S})\right)$, for each $0\leqslant r\leqslant n$. Consider first $F^ra_0\prod_{i=1}^{n-r}[D,a_i]\otimes T\in F^r\pi\left(\Omega^{n-r}(\mathcal{A})\right)\otimes
\mathcal{S}\,$, where $1\leqslant r\leqslant n-1$. By Lemma ($\,$\ref{finite support 1}$\,$), one can write $T=\sum_k T_{0k}\prod_{i=1}^r [N,T_{ik}]$. Let $I_{(0k)}$ be the infinite matrix having an identity block matrix in
top left most corner of order same as that of $T_{0k}$ and zero elsewhere. Then,
\begin{eqnarray*}
&  & \sum_k(a_0\otimes T_{0k})\left(\prod_{i=1}^{n-r}[\varSigma^2 D,a_i\otimes I_{(0k)}]\right) \left(\prod_{j=1}^r[\varSigma^2 D,1\otimes T_{jk}]\right)\\
&  & = \sum_k (a_0\otimes T_{0k})\prod_{i=1}^{n-r}\left([D,a_i]\otimes I_{(0k)}\right) \left(F^r\otimes \prod_{j=1}^{r}[N,T_{jk}]\right)\\
&  & = \sum_k F^r a_0\left(\prod_{i=1}^{n-r}[D,a_i]\right)\otimes T_{0k}\left(\prod_{j=1}^r[N,T_{jk}]\right)\\
&  & = F^r a_0\prod_{i=1}^{n-r}[D,a_i]\otimes T.
\end{eqnarray*}
For $r=0$, observe that $\,a_0\prod_{i=1}^{n}[D,a_i]\otimes T = (a_0\otimes T)\prod_{i=1}^{n}[\varSigma^2D,a_i\otimes I_{(T)}]\,$, where $\,I_{(T)}\,$ denotes the infinite matrix having an identity block matrix in top left
most corner of order same as that of $T$ and zero elsewhere. Finally for $r=n$, $\,F^na\otimes T =  \sum_k(a\otimes T_{0k})\left(\prod_{i=1}^n[\varSigma^2 D,1\otimes T_{ik}]\right)\,$ where $\,T=\sum_k T_{0k}\prod_{i=1}^n
[N,T_{ik}]\,$ (by Lemma ($\,$\ref{finite support 1}$\,$)).
\end{proof}

\begin{lemma}\label{1 forms involving A,S}
$\pi\left(dJ_0^1(\mathcal{A}\otimes\mathcal{S})\right) = \pi\left(dJ_0^1(\mathcal{A})\right)\otimes \mathcal{S} + F\pi\left(\Omega^1(\mathcal{A})\right)\otimes\mathcal{S}+ \mathcal{A}\otimes \mathcal{S}$.
\end{lemma}

\begin{proof}
In terms of elementary matrices $(e_{ij})$ arbitrary element of $\pi\left(dJ_0^1(\mathcal{A}\otimes \mathcal{S})\right)$ looks like 
\begin{eqnarray*}
\sum [\varSigma^2 D,a_0\otimes T_0][\varSigma^2 D,a_1\otimes T_1] & =  & \sum \sum_{i,s} \textbf{\{}\sum_j [D,a_{0ij}][D,a_{1js}] + Fa_{0ij}[D,a_{1js}](i-j)\\
&   & \quad\quad+ a_{0ij}a_{1js}(i-j)(j-s) + F[D,a_{0ij}]a_{1js}(j-s)\textbf{\}}\otimes e_{is}
\end{eqnarray*}
such that for each $\,i$ and $s\,$ we have
\begin{eqnarray}\label{equation for J^1}
 \sum \sum_j a_{0ij}[D,a_{1js}]+\sum_j Fa_{0ij}a_{1js}(j-s) & = & 0\,.
\end{eqnarray}
Consider the following equations
\begin{eqnarray}\label{equations 1}
\xi & = & \sum \sum_j a_{0ij}[D,a_{1js}]\,,
\end{eqnarray}
\begin{eqnarray}\label{equations 2}
\,\,\,\,\,\,\eta & = & \sum \sum_j a_{0ij}a_{1js}(j-s)\,.
\end{eqnarray}
Hence, $\xi+F\eta=0$ by equation ($\,$\ref{equation for J^1}$\,$). Let $n$ be any natural number. For each $i$ and $s\,$, consider 
\begin{center}
$
\begin{cases}
 \begin{array}{lcl}
  a_{0,i,s+n} = -1\\
  a_{1,s+n,s} = a
 \end{array}
\end{cases}
\begin{cases}
 \begin{array}{lcl}
  a_{0,i,s+n+1} = 1\\
  a_{1,s+n+1,s} = a
 \end{array}
\end{cases}
\begin{cases}
 \begin{array}{lcl}
  a_{0,i,s+n+2} = 1\\
  a_{1,s+n+2,s} = a
 \end{array}
\end{cases}
\begin{cases}
 \begin{array}{lcl}
  a_{0,i,s+n+3} = 1\\
  a_{1,s+n+3,s} = -a\,.
 \end{array}
\end{cases}
$
\end{center}
One easily checks that $\,\xi=0$ in equation ($\,$\ref{equations 1}$\,$) and $\,\eta=0$ in equation ($\,$\ref{equations 2}$\,$) for these four pairs together and hence these pairs can produce infinitely many 
solutions to the equation $\,\xi+F\eta=0$. We can now conclude that arbitrary $a\otimes e_{is}$ lies in $\pi(dJ_0^1(\mathcal{A}\otimes \mathcal{S}))$, for each $i$ and $s$ and for any $a\in \mathcal{A}$.
We will now show that any $Fa[D,b]\otimes e_{is}$ lies in $\pi\left(dJ_0^1(\mathcal{A}\otimes \mathcal{S})\right)$, for each $i$ and $s$. For any natural number $m$ and for each $i$ and $s\,$, consider
\begin{center}
$
\begin{cases}
 \begin{array}{lcl}
  a_{0,i,s+m} = a\\
  a_{1,s+m,s} = b
 \end{array}
\end{cases}
\begin{cases}
 \begin{array}{lcl}
  a_{0,i,s+m+1} = -a\\
  a_{1,s+m+1,s} = b
 \end{array}
\end{cases}
\begin{cases}
 \begin{array}{lcl}
  a_{0,i,s+m+2} = (1/(m+2))ab\\
  a_{1,s+m+2,s} = 1\,.
 \end{array}
\end{cases}
$
\end{center}
Again one checks that $\,\xi=0$ in equations ($\,$\ref{equations 1}$\,$) and $\,\eta=0$ in equation ($\,$\ref{equations 2}$\,$) and one gets $\,2Fa[D,b]\otimes e_{is} + ab(m-1)\otimes e_{is}$ as an element of
$\pi\left(dJ_0^1(\mathcal{A}\otimes \mathcal{S})\right)$. Hence $\,Fa[D,b]\otimes e_{is}$ lies in $\pi\left(dJ_0^1(\mathcal{A}\otimes \mathcal{S})\right)$, for each $i$ and $s$. We will now show that $\pi\left(dJ_0^1
(\mathcal{A})\right)\otimes\mathcal{S}\subseteq \pi\left(dJ_0^1(\mathcal{A}\otimes \mathcal{S})\right)$. Choose any $\,\sum_k [D,a_{0k}][D,a_{1k}]\otimes T_k\in \pi\left(dJ_0^1(\mathcal{A})\right)\otimes \mathcal{S}$. Then
$a_{0k}[D,a_{1k}]=0$ for each $k$. Let $I_{(k)}$ be the infinite matrix having an identity block matrix in top left most corner of order same as that of $T_{k}$ and zero elsewhere. Then $\,[\varSigma^2D,a_{0k}\otimes
I_{(k)}][\varSigma^2D,a_{1k}\otimes I_{(k)}]\in \pi\left(dJ_0^1(\mathcal{A}\otimes \mathcal{S})\right)$ for each $k$ and hence, $\,(1\otimes T_k)[\varSigma^2D,a_{0k}\otimes I_{(k)}][\varSigma^2D,a_{1k}\otimes I_{(k)}]\in
\pi\left(dJ_0^1(\mathcal{A}\otimes \mathcal{S})\right)$ for each $k$. Now observe that $$\sum_k [D,a_{0k}][D,a_{1k}]\otimes T_k=\sum_k (1\otimes T_k)[\varSigma^2D,a_{0k}\otimes I_{(k)}][\varSigma^2D,a_{1k}\otimes I_{(k)}]
\,.$$ This proves the inclusion `$\supseteq$'. Now Lemma ($\,$\ref{n forms involving A,S}$\,$) shows that $\pi\left(dJ_0^1(\mathcal{A}\otimes \mathcal{S})\right)\subseteq\pi\left(\Omega^2(\mathcal{A})\right)\otimes \mathcal{S}
+ F\pi\left(\Omega^1(\mathcal{A})\right)\otimes\mathcal{S}+ \mathcal{A}\otimes \mathcal{S}$. Finally, the fact that $[D,\mathcal{A}]\subseteq\mathcal{A}\otimes\mathcal{E}nd_{\mathcal{A}}(\mathcal{H}^\infty)\subseteq
\mathcal{E}nd_{\mathbb{C}}(\mathcal{H}^\infty)$ and $F\mathcal{A}\cap\mathcal{A}=\{0\}$ implies the inclusion `$\subseteq$' by equation ($\,$\ref{equation for J^1}$\,$).
\end{proof}

\begin{lemma}\label{general element}
For all $n\geq 1$, $F^{n+1} a\otimes e_{ij} \in \pi\left(dJ_0^n(\mathcal{A}\otimes \mathcal{S})\right)$ for any $a\in \mathcal{A}$ and each $i$ and $j$.
\end{lemma}

\begin{proof}
The $n=1$ case has been addressed in Lemma ($\,$\ref{1 forms involving A,S}$\,$). Let's take $n\geq 2$. Arbitrary element of $\pi\left(dJ_0^n(\mathcal{A}\otimes \mathcal{S})\right)$ looks like
\begin{eqnarray}\label{Arbitrary element}
\sum \sum_{i_1,i_{n+2}}\left( \sum_{i_2,\ldots,i_{n+1}} \prod_{j=1}^{n+1}[D,a_{ji_ji_{j+1}}] + \sum_{t=1}^{n+1} F^t(\prod_{j=1}^{n+1}[D,a_{ji_ji_{j+1}}])^{(t)} \right)\otimes e_{i_1i_{n+2}}\,,
\end{eqnarray}
where $\left(\prod_{j=1}^{n+1}[D,a_{ji_ji_{j+1}}]\right)^{(t)}$ is the following expression
$$\sum_{1\leq r_1<r_2<\ldots<r_t}^{n+1} [D,a_{1i_1i_2}]\ldots\hat{[D,a_{r_1i_{r_1}i_{r_1+1}}]}\ldots \hat{[D,a_{r_2i_{r_2}i_{r_2+1}}]}\ldots \hat{[D,a_{r_ti_{r_t}i_{r_t+1}}]}\ldots [D,a_{(n+1)i_{n+1}i_{n+2}}]$$
with $\, \hat{[D,a_{ri_ri_{r+1}}]} = (i_r-i_{r+1})a_{ri_ri_{r+1}}\, $ (total number of $\, \verb!^!\, $ appears in each summand of the summation $\sum_{1\leq r_1<r_2<\ldots<r_t}^n$ is exactly $\, t\, $)~; such that
\begin{eqnarray*}
&  & \sum \sum_{i_2,\ldots,i_{n+1}}\textbf{\{}\,a_{1i_1i_2}\prod_{j=2}^{n+1}[D,a_{ji_ji_{j+1}}] + \sum_{t=2,\, t\, even}^{n} a_{1i_1i_2}\left(\prod_{j=2}^{n+1}[D,a_{ji_ji_{j+1}}]\right)^{(t)}\\
&  & \quad\quad\quad\quad\quad +\sum_{t=3,\, t\, odd}^{n} Fa_{1i_1i_2}\left(\prod_{j=2}^{n+1}[D,a_{ji_ji_{j+1}}]\right)^{(t)}\textbf{\}} = 0
\end{eqnarray*}
for each $i_1$ and $i_{n+2}\, $. Here $\left(\prod_{j=2}^{n+1}[D,a_{ji_ji_{j+1}}]\right)^{(t)}$ is the same expression as $\left(\prod_{j=1}^{n+1}[D,a_{ji_ji_{j+1}}]\right)^{(t)}$ except for the fact that there is no $r_1$
present i,e. the summation will be over $r_2,\ldots,r_t$ and for $t=1$ this term is zero. Consider
\begin{eqnarray}\label{equations 3}
\xi & = & \sum \sum_{i_2,\ldots,i_{n+1}} a_{1i_1i_2} \prod_{j=2}^{n+1}[D,a_{ji_ji_{j+1}}]\,,
\end{eqnarray}
\begin{eqnarray}\label{equations 4}
\quad\quad\quad\eta & = & \sum \sum_{i_2,\ldots,i_{n+1}}\textbf{\{}\sum_{t=2,\, t\, even}^{n} a_{1i_1i_2}(\prod_{j=2}^{n+1}[D,a_{ji_ji_{j+1}}])^{(t)}+\sum_{t=3,\, t\, odd}^{n} Fa_{1i_1i_2}(\prod_{j=2}^{n+1}[D,a_{ji_ji_{j+1}}])^{(t)}\}\,.
\end{eqnarray}
Hence, $\,\xi+\eta=0$. For each $i_1$ and $i_{n+2}\, $, consider
\begin{center}
$
\begin{cases}
 \begin{array}{lcl}
  a_{1,i_1,i_1+m} = a\\
  a_{s,i_1+(s-1)m,i_1+sm} = -1/m\, \, ;\, \, \forall\, \, 2\leq s\leq n\\
  a_{n+1,i_1+nm,i_{n+2}} = 1/(i_1+nm-i_{n+2})
 \end{array}
\end{cases}
$
\end{center}
and
\begin{center}
$
\begin{cases}
 \begin{array}{lcl}
  a_{1,i_1,i_1+m+1} = -a\\
  a_{s,i_1+(s-1)m+1,i_1+sm+1} = -1/m\, \, ;\, \, \forall\, \, 2\leq s\leq n\\
  a_{n+1,i_1+nm+1,i_{n+2}} = 1/(i_1+nm+1-i_{n+2})
 \end{array}
\end{cases}
$
\end{center}
Here $\,m\,$ is a natural number s.t. $\, i_1+nm+1-i_{n+2}\, $ and $\, i_1+nm-i_{n+2}\, $ both are nonzero. Note that infinitely many such $\,m\,$ can be found for given $i_1,i_{n+2},n$. The term 
$\, \sum_{i_2,\ldots,i_{n+1}} a_{1i_1i_2}\prod_{j=2}^{n+1} (a_{ji_ji_{j+1}}(i_j-i_{j+1}))\, $ becomes zero for above choice and these pairs satisfy $\,\xi=0\,$ in equation ($\,$\ref{equations 3}$\,$) and $\,\eta=0\,$ in equation
($\,$\ref{equations 4}$\,$). Existence of infinitely many natural numbers $\,m\,$ gives us infinitely many solutions to the equation $\,\xi+\eta=0$. The only surviving term in the expression ($\,$\ref{Arbitrary element}$\,$)
for these solution is the term $\sum_{i_2,\ldots,i_{n+1}} F^{n+1} \left(\prod_{k=1}^{n+1} a_{ki_ki_{k+1}}\right) \prod_{j=1}^{n+1} (i_j-i_{j+1})$ (when $t=n+1$), which is equal to $F^{n+1} a\,$ for each $i_1$ and $i_{n+2}$.
All the other terms become zero because of the existence of commutators (except for $t=n$, which also vanishes for our choice $\,a_{1,i_1,i_1+m} = a,\,a_{1,i_1,i_1+m+1} = -a$). This justifies our claim.
\end{proof}
\medskip

\begin{lemma}\label{to refer in thm}
$\sum_{j=0}^n F^{n+1-j}\pi(\Omega^j(\mathcal{A}))\otimes \mathcal{S} \subseteq \pi\left(dJ_0^n(\mathcal{A}\otimes \mathcal{S})\right)$, for all $n\geq 1$.
\end{lemma}

\begin{proof}
We use induction. We have seen that $F\pi(\Omega^1(\mathcal{A}))\otimes \mathcal{S} + \mathcal{A}\otimes \mathcal{S} \subseteq \pi(dJ_0^1(\mathcal{A}\otimes \mathcal{S}))$ (Lemma ($\,$\ref{1 forms involving A,S}$\,$)).
Assume this is true at the $k$-th stage. To prove for the $k+1$-th stage we use Lemma ($\,$\ref{general element}$\,$). Let $\xi=\sum_{j=0}^{k+1}F^{k+2-j}\xi_{j}\in \sum_{j=0}^{k+1} F^{k+2-j}\pi(\Omega^j(\mathcal{A}))\otimes
\mathcal{S}$. Lemma ($\,$\ref{general element}$\,$) shows that $F^{k+2}\xi_{0}\in \pi\left(dJ_0^{k+1}(\mathcal{A}\otimes \mathcal{S})\right)$. To prove $\,\xi-F^{k+2}\xi_{0}\in\pi\left(dJ_0^{k+1}(\mathcal{A}\otimes \mathcal{S}
)\right)$, it is enough to show that for each $i$ and $j\, $, if $\pi(\omega)\otimes e_{ij} \in \pi(dJ_0^k(\mathcal{A}\otimes \mathcal{S}))\,$ then $\,\pi(\omega)[D,a]\otimes e_{ij} \in \pi(dJ_0^{k+1}(\mathcal{A}\otimes
\mathcal{S}))$ for any $a \in \mathcal{A}\, $. Let
\begin{eqnarray*}
\pi(\omega)\otimes e_{ij} & = & \sum \prod_{m=0}^k[\varSigma^2 D,x_m]
\end{eqnarray*}
such that
\begin{eqnarray}\label{Equation}
\sum x_0\prod_{m=1}^k[\varSigma^2 D,x_m] & = & 0.
\end{eqnarray}
for all $x_m \in \mathcal{A}\otimes \mathcal{S}\, $. Now 
\begin{eqnarray*}
(\pi(\omega)\otimes e_{ij})([D,a]\otimes e_{jj}) & = & \left(\sum \prod_{m=0}^k[\varSigma^2 D,x_m]\right)([D,a]\otimes e_{jj})\\
                                                 & = & \sum \prod_{m=0}^k[\varSigma^2 D,x_m][\varSigma^2 D,a\otimes e_{jj}]\,,
\end{eqnarray*}
since $[D,a]\otimes e_{jj} = [\varSigma^2 D,a\otimes e_{jj}]$. If $\pi(\omega)[D,a]\otimes e_{ij}$ has to be in $\pi(dJ_0^{k+1}(\mathcal{A}\otimes \mathcal{S}))$ then
\begin{eqnarray*}
\sum x_0\prod_{m=1}^k[\varSigma^2 D,x_m][\varSigma^2 D,a\otimes e_{jj}] & = & 0
\end{eqnarray*}
should hold. But this is clear from equation ($\,$\ref{Equation}$\,$).
\end{proof}

\begin{lemma}\label{containment of top degree}
For all $n\geq 1$, $\,\pi\left(dJ_0^n(\mathcal{A})\right)\otimes \mathcal{S}\subseteq \pi\left(dJ_0^n(\mathcal{A}\otimes \mathcal{S})\right)\,$.
\end{lemma}

\begin{proof}
The $n=1$ case has been addressed in Lemma ($\,$\ref{1 forms involving A,S}$\,$). Let's take $n\geq 2$. Now choose any $\,\sum_k \prod_{j=0}^n[D,a_{jk}]\otimes T_k\in \pi\left(dJ_0^n(\mathcal{A})\right)\otimes \mathcal{S}$.
Then $\,a_{0k}\prod_{j=1}^n[D,a_{jk}]=0$ for each $k$. Let $I_{(k)}$ be the infinite matrix having an identity block matrix in top left most corner of order same as that of $T_{k}$ and zero elsewhere. Then $\,\prod_{j=0}^n
[\varSigma^2D,a_{jk}\otimes I_{(k)}]\in \pi\left(dJ_0^n(\mathcal{A}\otimes \mathcal{S})\right)$ for each $k$ and hence, $\,(1\otimes T_k)\prod_{j=0}^n[\varSigma^2D,a_{jk}\otimes I_{(k)}]\in \pi\left(dJ_0^n(\mathcal{A}
\otimes \mathcal{S})\right)$ for each $k$. Now observe that $$\sum_k \prod_{j=0}^n[D,a_{jk}]\otimes T_k=\sum_k (1\otimes T_k)\prod_{j=0}^n[\varSigma^2D,a_{jk}\otimes I_{(k)}]\,,$$ which completes the proof.
\end{proof}

\begin{lemma}\label{the denominators}
$\pi\left(dJ_0^n(\mathcal{A}\otimes\mathcal{S})\right)=\pi\left(dJ_0^n(\mathcal{A})\right)\otimes\mathcal{S}+\sum_{r=0}^nF^{n+1-r}\pi\left(\Omega^r(\mathcal{A})\right)\otimes\mathcal{S}$, for all $n\geq 1$.
\end{lemma}

\begin{proof}
Since $\pi\left(dJ_0^n(\mathcal{A}\otimes\mathcal{S})\right)\subseteq\pi\left(\Omega^{n+1}(\mathcal{A}\otimes\mathcal{S})\right)$, Lemma ($\,$\ref{n forms involving A,S}$\,$) says the following $$\pi\left(dJ_0^n(\mathcal{A}
\otimes\mathcal{S})\right)\subseteq \pi\left(\Omega^{n+1}(\mathcal{A})\right)\otimes\mathcal{S}+\sum_{r=0}^nF^{n+1-r}\pi\left(\Omega^r(\mathcal{A})\right)\otimes\mathcal{S}.$$ Now Lemma ($\,$\ref{general element}$\,$), ($\,$\ref{to refer in thm}$\,$) and 
($\,$\ref{containment of top degree}$\,$) shows that $$\pi\left(dJ_0^n(\mathcal{A})\right)\otimes \mathcal{S}+\sum_{r=0}^nF^{n+1-r}\pi\left(\Omega^r(\mathcal{A})\right)\otimes\mathcal{S}\subseteq\pi\left(dJ_0^n(\mathcal{A}
\otimes\mathcal{S})\right).$$ We need to show that $$\pi\left(dJ_0^n(\mathcal{A}\otimes\mathcal{S})\right)\subseteq \pi\left(dJ_0^n(\mathcal{A})\right)\otimes \mathcal{S}+\sum_{r=0}^nF^{n+1-r}\pi\left(\Omega^r(\mathcal{A})
\right)\otimes\mathcal{S}.$$ We use induction on $n\,$. Lemma ($\,$\ref{1 forms involving A,S}$\,$) gives the basis step of the induction and suppose that $$\pi\left(dJ_0^{n-1}(\mathcal{A}\otimes\mathcal{S})\right)=
\pi\left(dJ_0^{n-1}(\mathcal{A})\right)\otimes\mathcal{S}+\sum_{r=0}^{n-1}F^{n-r}\pi\left(\Omega^r(\mathcal{A})\right)\otimes\mathcal{S}.$$ Recall that for algebra $\mathcal{B}$ and for all $n\geq 1\,,$
\begin{eqnarray*}
 \Omega^n(\mathcal{B}) & = & \underbrace{\Omega^1(\mathcal{B})\otimes_{\mathcal{B}}\ldots\ldots\otimes_{\mathcal{B}}\Omega^1(\mathcal{B})}_{n\,\,times}\\
                       & = & \Omega^1(\mathcal{B})\otimes_{\mathcal{B}}\Omega^{n-1}(\mathcal{B})\,.
\end{eqnarray*}
Hence $\,J_0^n=J_0^1\otimes\Omega^{n-1}+\Omega^1\otimes J_0^{n-1}$. Since $\,d\,$ satisfies graded Leibniz rule, we have $$dJ_0^n\subseteq (dJ_0^1).\,\Omega^{n-1}+J_0^1.\,(d\Omega^{n-1})+(d\Omega^1).\,J_0^{n-1}+\Omega^1.\,
(dJ_0^{n-1})$$ (recall the graded product on $\,\Omega^\bullet\,$) and hence applying the algebra homomorphism $\,\pi\,$ we get $\,\pi(dJ_0^n)\subseteq\pi(dJ_0^1)\pi(\Omega^{n-1})+\pi(\Omega^1)\pi(dJ_0^{n-1})$.
Since $J^\bullet$ is a graded ideal in $\,\Omega^\bullet$ we have
\begin{eqnarray*}
 \pi\left(dJ_0^n(\mathcal{A}\otimes\mathcal{S})\right) & \subseteq & \pi\left(dJ_0^1(\mathcal{A}\otimes\mathcal{S})\right)\pi\left(\Omega^{n-1}(\mathcal{A}\otimes\mathcal{S})\right)+\pi\left(\Omega^1(\mathcal{A}\otimes
\mathcal{S})\right)\pi\left(dJ_0^{n-1}(\mathcal{A}\otimes\mathcal{S})\right)\\
& = & \left(\pi\left(dJ_0^1(\mathcal{A})\right)\otimes\mathcal{S}+F\pi\left(\Omega^1(\mathcal{A})\right)\otimes\mathcal{S}+\mathcal{A}\otimes\mathcal{S}\right)\left(\sum_{r=0}^{n-1} F^r\pi(\Omega^{n-1-r}(\mathcal{A}))\otimes
\mathcal{S}\right)\\
&   & +\left(\pi(\Omega^1(\mathcal{A}))\otimes\mathcal{S}+F\mathcal{A}\otimes\mathcal{S}\right)\left(\pi\left(dJ_0^{n-1}(\mathcal{A})\right)\otimes\mathcal{S}+\sum_{r=0}^{n-1} F^{n-r}\pi(\Omega^r(\mathcal{A}))\otimes
\mathcal{S}\right)\\
& = & \sum_{r=0}^{n-1}F^r\pi(dJ_0^{n-r}(\mathcal{A}))\otimes\mathcal{S}+\sum_{r=0}^{n-1}F^{r+1}\pi(\Omega^{n-r}(\mathcal{A}))\otimes\mathcal{S}\\
&   & +\sum_{r=0}^{n-1}F^r\pi(\Omega^{n-1-r}(\mathcal{A}))\otimes\mathcal{S}+\pi\left(dJ_0^n(\mathcal{A})\right)\otimes\mathcal{S}+F\pi\left(dJ_0^{n-1}(\mathcal{A})\right)\otimes\mathcal{S}\\
&   & +\sum_{r=0}^{n-1}F^{n+1-r}\pi\left(\Omega^r(\mathcal{A})\right)\otimes\mathcal{S}+\sum_{r=1}^nF^{n+1-r}\pi\left(\Omega^r(\mathcal{A})\right)\otimes\mathcal{S}\\
& = & \pi\left(dJ_0^n(\mathcal{A})\right)\otimes \mathcal{S}+\sum_{r=0}^nF^{n+1-r}\pi\left(\Omega^r(\mathcal{A})\right)\otimes\mathcal{S}.
\end{eqnarray*}
Here the first equality follows from Lemmas ($\,$\ref{1 forms involving A,S}$\,)\,,\,(\,$\ref{n forms involving A,S}$\,$) and the induction hypothesis.
\end{proof}

\begin{remark}
Note that the second condition, i,e. $[D,\mathcal{A}]\subseteq\mathcal{A}\otimes\mathcal{E}nd_{\mathcal{A}}(\mathcal{H}^\infty)$, is needed only for Lemmas $(\,\ref{1 forms involving A,S}\,,\,\ref{the denominators}\,)$.
\end{remark}

\begin{theorem}\label{final thm}
For the spectral triple $\, \left(\varSigma^2 \mathcal{A},\varSigma^2 \mathcal{H},\varSigma^2 D\right)\, $, we have
\begin{enumerate}
 \item $\Omega_{\varSigma^2 D}^1\left(\varSigma^2 \mathcal{A}\right) \cong \Omega_D^1(\mathcal{A})\otimes\mathcal{S}\bigoplus \varSigma^2 \mathcal{A}\,$.
 \item $\Omega_{\varSigma^2 D}^n\left(\varSigma^2 \mathcal{A}\right) \cong \Omega_D^n(\mathcal{A})\otimes \mathcal{S}\,$, $\,$ for all $n\geq 2\, $.
 \item The differential $\, \, \delta^0: \varSigma^2 \mathcal{A}\longrightarrow \Omega_{\varSigma^2 D}^1\left(\varSigma^2 \mathcal{A}\right)\, $ is given by,
  \begin{center}
   $a\otimes T + f \longmapsto [D,a]\otimes T \bigoplus \left(a\otimes[N,T]+f^\prime\right)$.
  \end{center}
  \item The differential $\, \, \delta^1: \Omega_{\varSigma^2 D}^1\left(\varSigma^2 \mathcal{A}\right) \longrightarrow \Omega_{\varSigma^2 D}^2\left(\varSigma^2 \mathcal{A}\right)\, $ is given by,
  \begin{center}
   $\delta^1|_{\Omega_D^1(\mathcal{A})\otimes \mathcal{S}}=d^1\otimes 1\,\,$ and $\,\,\delta^1|_{\varSigma^2 \mathcal{A}}=0$.
  \end{center}
 \item The differential $\, \, \delta^n: \Omega_{\varSigma^2 D}^n\left(\varSigma^2 \mathcal{A}\right) \longrightarrow \Omega_{\varSigma^2 D}^{n+1}\left(\varSigma^2 \mathcal{A}\right)\, $ is given by,
  \begin{center}
   $\delta^n=d^n\otimes 1$
  \end{center}
for all $n\geq 2\,$.
\end{enumerate}
Here $\, d:\Omega_D^\bullet(\mathcal{A})\longrightarrow\Omega_D^{\bullet+1}(\mathcal{A})$ is the differential of the Connes' complex.
\end{theorem}

\begin{proof}
\begin{enumerate}
\item Recall from Lemma ($\,$\ref{n forms involving A,S}$\,$), $\,\Omega_{\varSigma^2 D}^1(\mathcal{A}\otimes \mathcal{S})\cong \Omega_D^1(\mathcal{A})\otimes \mathcal{S}+F\mathcal{A}\otimes \mathcal{S}$. Since
$\mathcal{A}\otimes\mathcal{E}nd_{\mathcal{A}}(\mathcal{H}^\infty)\subseteq\mathcal{E}nd_{\mathbb{C}}(\mathcal{H}^\infty)$ by the map $\,a\otimes g\longmapsto a\circ g\,$, $F\mathcal{A}$ can be embedded in
$\,F\mathcal{A}\otimes\mathcal{E}nd_{\mathcal{A}}(\mathcal{H}^\infty)\subseteq \mathcal{E}nd_{\mathbb{C}}(\mathcal{H}^\infty)$ by the map $\,Fa\longmapsto Fa\otimes I$. Now $\,[D,\mathcal{A}]\subseteq\mathcal{E}nd
_{\mathbb{C}}(\mathcal{H}^\infty)$ and $\,F\mathcal{A}\cap\mathcal{A}=\{0\}$ gives the direct sum. Finally, use the fact that $F\mathcal{A}\otimes\mathcal{S}\cong\mathcal{A}\otimes\mathcal{S}$ and Proposition
($\,$\ref{imp theorem}$\,$) to conclude Part ($1$).
\item For all $n\geq 2\,,$ we have
\begin{eqnarray*}
&    & \Omega_{\varSigma^2 D}^n\left(\varSigma^2 \mathcal{A}\right)\\
& \cong & \frac{\pi\left(\Omega^n(\varSigma^2\mathcal{A})\right)}{\pi\left(dJ_0^{n-1}(\varSigma^2 \mathcal{A})\right)}\\
& \cong & \frac{\pi\left(\Omega^n(\mathcal{A}\otimes\mathcal{S})\right)}{\pi\left(dJ_0^{n-1}(\mathcal{A}\otimes\mathcal{S})\right)}\quad\quad by\,\, Proposition\,(\,\ref{imp theorem}\,)\\
& \cong & \frac{\sum_{r=0}^nF^r\pi\left(\Omega^{n-r}(\mathcal{A})\right)\otimes\mathcal{S}}{\pi\left(dJ_0^{n-1}(\mathcal{A})\right)\otimes\mathcal{S}+\sum_{r=0}^{n-1}F^{n-r}\pi\left(\Omega^{r}(\mathcal{A})\right)\otimes
\mathcal{S}}\quad\quad by\,\, Lemma\,(\,\ref{n forms involving A,S}\,)\,,\,(\,\ref{the denominators}\,)\\
& \cong & \frac{\pi\left(\Omega^n(\mathcal{A})\right)\otimes\mathcal{S}+\sum_{r=1}^nF^r\pi\left(\Omega^{n-r}(\mathcal{A})\right)\otimes\mathcal{S}}{\pi\left(dJ_0^{n-1}(\mathcal{A})\right)\otimes\mathcal{S}+\sum_{r=0}^{n-1}
F^{n-r}\pi\left(\Omega^{r}(\mathcal{A})\right)\otimes\mathcal{S}}\\
& \cong & \frac{\pi\left(\Omega^n(\mathcal{A})\right)\otimes\mathcal{S}+\sum_{r=0}^{n-1}F^{n-r}\pi\left(\Omega^{r}(\mathcal{A})\right)\otimes\mathcal{S}}{\pi\left(dJ_0^{n-1}(\mathcal{A})\right)\otimes\mathcal{S}+
\sum_{r=0}^{n-1}F^{n-r}\pi\left(\Omega^{r}(\mathcal{A})\right)\otimes\mathcal{S}}\\
& \cong & \frac{\pi\left(\Omega^n(\mathcal{A})\right)\otimes\mathcal{S}}{\pi\left(dJ_0^{n-1}(\mathcal{A})\right)\otimes\mathcal{S}}\\
& \cong & \Omega_D^n(\mathcal{A})\otimes\mathcal{S}
\end{eqnarray*}
\item Obvious since $[\varSigma^2D,a\otimes T+f]=[D,a]\otimes T+Fa\otimes [N,T]+f^\prime$.
\item Take arbitrary $(a_0[D,a_1]\otimes T\,,\,b\otimes S+f)\in\Omega_D^1(\mathcal{A})\otimes\mathcal{S}\bigoplus \varSigma^2\mathcal{A}\,$. Using Lemma ($\,$\ref{finite support 1}$\,$) we have $S=\sum S_0[N,S_1]$ and
Proposition ($\,$\ref{forms for circle}$\,$) implies $f=\sum f_0f_1^\prime\,$. Now, as an element of $\,\Omega_{\varSigma^2 D}^1(\varSigma^2 \mathcal{A})$,
\begin{eqnarray*}
&   & (a_0[D,a_1]\otimes T\,,\,b\otimes S+f)\\
& = & (a_0\otimes T)[\varSigma^2D,a_1\otimes I_{(T)}]+\sum(b\otimes S_0)[\varSigma^2D,1\otimes S_1]+\sum(1\otimes f_0)[\varSigma^2D,1\otimes f_1]
\end{eqnarray*}
where $I_{(T)}$ is the identity block matrix of order same as that of $T$. Hence,
\begin{eqnarray*}
&   & \delta^1((a_0[D,a_1]\otimes T\,,\,b\otimes S+f))\\
& = & \textbf{(}[\varSigma^2D,a_0\otimes T][\varSigma^2D,a_1\otimes I_{(T)}]+\sum[\varSigma^2D,b\otimes S_0][\varSigma^2D,1\otimes S_1]\\
&   & \quad+\sum[\varSigma^2D,1\otimes f_0][\varSigma^2D,1\otimes f_1]\textbf{)}+\pi(dJ_0^1(\varSigma^2\mathcal{A})),
\end{eqnarray*}
as an element of $\,\Omega_{\varSigma^2D}^2(\varSigma^2\mathcal{A})\cong\frac{\pi(\Omega^2(\varSigma^2\mathcal{A}))}{\pi(dJ_0^1(\varSigma^2\mathcal{A}))}$. Now, by Part $(2)$, we finally get $$\delta^1((a_0[D,a_1]\otimes
T\,,\,b\otimes S+f))\,=\,\left([D,a_0][D,a_1]+\pi(dJ_0^1(\mathcal{A}))\right)\otimes T,$$ as an element of $\Omega_D^2(\mathcal{A})\otimes\mathcal{S}$.
\item Follows similarly as Part $(4)$.
\end{enumerate}
\end{proof}

Now we want to iterate this Theorem and Proposition ($\,$\ref{justification for assumption}$\,$) guarantees that we are allowed to do so. Let $\,k\geq 1$ and $\varSigma^{2k}\mathcal{A}=\varSigma^2(\varSigma^{2(k-1)}
\mathcal{A})$. We put the convention $\varSigma^0\mathcal{A}=\mathcal{A}$ and $\varSigma^0D=D$. Let $\,F^{(k)}$ be the sign of the operator $\varSigma^{2(k-1)}D$, acting on the Hilbert space $\mathcal{H}\otimes\ell^2
(\mathbb{N})^{\otimes k-1}$. Then $\varSigma^{2k}D=\varSigma^{2(k-1)}D\otimes I+F^{(k)}\otimes N$ and $F^{(k)}=F\otimes 1^{\otimes k-1}$. Any element $\,\varSigma^{2k}a\,$ of $\,\varSigma^{2k}\mathcal{A}\,$ is of the form
$\varSigma^{2(k-1)}a\otimes T_{(k-1)}+f_{(k-1)}$ where $\varSigma^{2(k-1)}a\in\varSigma^{2(k-1)}\mathcal{A}\,,\,T_{(k-1)}\in\mathcal{S}$ and $f_{(k-1)}\in\mathbb{C}[z,z^{-1}]$. Using this functional equation one can write
$\varSigma^{2k}a$ in terms of elements only from $\mathcal{A},\,\mathcal{S}$ and $\mathbb{C}[z,z^{-1}]$.

\begin{corollaryl}
For the spectral triple $\,\left(\varSigma^{2k}\mathcal{A},\varSigma^{2k}\mathcal{H},\varSigma^{2k}D\right),\,k\geq 1$, we have
\begin{enumerate}
 \item $\Omega_{\varSigma^{2k}D}^1\left(\varSigma^{2k}\mathcal{A}\right)\cong\Omega_D^1(\mathcal{A})\otimes\mathcal{S}^{\otimes k}\bigoplus\bigoplus_{j=1}^k\left(\varSigma^{2j}\mathcal{A}\otimes\mathcal{S}^{\otimes(k-j)}\right)\,$.
 \item $\Omega_{\varSigma^{2k}D}^n\left(\varSigma^{2k}\mathcal{A}\right)\cong\Omega_D^n(\mathcal{A})\otimes\mathcal{S}^{\otimes k}\,$, $\,$ for all $n\geq 2\, $.
 \item The differential $\, \, \delta^0: \varSigma^{2k}\mathcal{A}\longrightarrow \Omega_{\varSigma^{2k}D}^1\left(\varSigma^{2k}\mathcal{A}\right)\, $ is given by,
  \begin{eqnarray*}
   &   & \varSigma^{2(k-1)}a\otimes T_{(k-1)} + f_{(k-1)}\quad\longmapsto\quad [D,a]\otimes T_{(0)}\otimes T_{(1)}\otimes\ldots\otimes T_{(k-1)}\,\,\bigoplus\\
   &   & \quad\quad\quad\quad\quad\quad\quad\quad\quad\quad\quad\quad\quad\quad\quad\,\,\left(\varSigma^{2(k-1)}a\otimes[N,T_{(k-1)}]+ f_{(k-1)}^\prime\right)\,\,\bigoplus\\
   &   & \quad\quad\quad\quad\quad\quad\quad\quad\quad\quad\quad\quad\quad\quad\quad\,\,\left(\bigoplus_{j=2}^k\left(\varSigma^{2(k-j)}a\otimes[N,T_{(k-j)}]+ f_{(k-j)}^\prime\right)\otimes Q_j\right)
  \end{eqnarray*} where $\,Q_j:=T_{(k-(j-1))}\otimes T_{(k-(j-2))}\otimes\ldots\otimes T_{(k-1)}\in\mathcal{S}^{\otimes(j-1)}\,$.
  \item The differential $\, \, \delta^1: \Omega_{\varSigma^{2k}D}^1\left(\varSigma^{2k}\mathcal{A}\right) \longrightarrow \Omega_{\varSigma^{2k}D}^2\left(\varSigma^{2k}\mathcal{A}\right)\, $ is given by,
  \begin{center}
   $\delta^1|_{\Omega_D^1(\mathcal{A})\otimes\mathcal{S}^{\otimes k}}=d^1\otimes 1^{\otimes k}\,\,$ and $\,\,\delta^1|_{\bigoplus_{j=1}^k\left(\varSigma^{2j}\mathcal{A}\otimes\mathcal{S}^{\otimes(k-j)}\right)}=0$.
  \end{center}
 \item The differential $\, \, \delta^n: \Omega_{\varSigma^{2k}D}^n\left(\varSigma^{2k}\mathcal{A}\right) \longrightarrow \Omega_{\varSigma^{2k}D}^{n+1}\left(\varSigma^{2k}\mathcal{A}\right)\, $ is given by,
  \begin{center}
   $\delta^n=d^n\otimes 1^{\otimes k}$
  \end{center}
for all $n\geq 2\,$.  
\end{enumerate}
Here $\, d:\Omega_D^\bullet(\mathcal{A})\longrightarrow\Omega_D^{\bullet+1}(\mathcal{A})$ is the differential of the Connes' complex at the $\,k=0\,$ level.
\end{corollaryl}

\begin{proof}
Note that $\,F^{(k)}\varSigma^{2(k-1)}\mathcal{A}\otimes\mathcal{S}\cong\varSigma^{2(k-1)}\mathcal{A}\otimes\mathcal{S}\,$ for all $\,k\geq 1$. Now the proof follows easily by induction on $\,k\,$ where Theorem
($\,$\ref{final thm}$\,$) is the basis step of the induction.
\end{proof}

\medskip

\section{Connection, Curvature for The Quantum Double Suspension}

Classical geometric objects like connection, curvature are extended to noncommutative set-up by Connes (\cite{4}) using the calculus $\Omega_D^\bullet$. These notions are meaningful whenever one has a spectral triple. In
this section, we discuss these notions on quantum double suspended spectral triple. We first recall the following necessary definitions from ((\cite{4}), Ch. 6).

Let $\,\mathcal{E}\,$ be a finitely generated projective(f.g.p)
module over $\mathcal{A}$, where $\mathcal{A}$ is a unital $\star$-algebra. We will always consider right modules in this section. Denote $\mathcal{E}^*$ to be the space of $\mathcal{A}$-linear maps from $\mathcal{E}$ to
$\mathcal{A}$. Clearly $\mathcal{E}^*$ is a right $\mathcal{A}$-module.

\begin{definition}
A {\it Hermitian} structure on $\mathcal{E}$ is an $\mathcal{A}$-valued positive-definite sesquilinear mapping 
$\langle \, \, , \, \rangle_{\mathcal{A}} $ such that, 
\begin{enumerate}
\item [(a)] $\langle \xi , \xi' \rangle_{\mathcal{A}}^* = \langle \xi' , \xi \rangle_{\mathcal{A}}\, , \, \, \, \forall \, \xi , \xi' \in \mathcal{E}$.
\item [(b)] $\langle \xi , \xi'. a \rangle_{\mathcal{A}} =  (\langle \xi , \xi' \rangle_{\mathcal{A}}) .a\, , \, \, \, \forall \, \xi , \xi' \in \mathcal{E},\, \, \forall \, a \in \mathcal{A}$.
\item [(c)] The map $\xi \longmapsto \Phi_\xi$ from $\mathcal{E}$ to $\mathcal{E}^*\,$, given by $\Phi_\xi(\eta) = \langle\xi,\eta\rangle_\mathcal{A}\, , \, \forall \eta \in \mathcal{E}\,$, 
gives a conjugate linear $\mathcal{A}$-module isomorphism between $\mathcal{E}$ and $\mathcal{E}^*$. This property will be referred as the self-duality of $\mathcal{E}$.
\end{enumerate}
\end{definition}

Any free $\mathcal{A}$-module $\mathcal{E}_0=\mathcal{A}^q$ has a Hermitian structure on it given by,
\begin{eqnarray*} 
\langle\, \xi , \eta\,\rangle_\mathcal{A} = \sum_{j=1}^q \xi_j^* \eta_j\,\, , \, \, \, \, \forall\, \xi = (\xi_1, \ldots, \xi_q) \, , \, \eta = (\eta_1, \ldots, \eta_q) \in \mathcal{E}_0.
\end{eqnarray*}
We refer it as the canonical Hermitian structure on $\mathcal{A}^q$.

\begin{remark}\label{1st remark}
 It is not known whether every f.g.p module $\mathcal{E}$ over $\mathcal{A}$ has a Hermitian structure on it. However, if we assume that $\mathcal{A}$ is spectrally invariant i,e. $\mathcal{A}$ is dense
subalgebra in a $C^*$-algebra $A$ and stable under holomorphic function calculus, then any f.g.p module $\mathcal{E}$ over $\mathcal{A}$ can be written as $p\mathcal{A}^n$ where $p\in M_n(\mathcal{A})$ is a self-adjoint
idempotent i,e. a projection, and hence has a Hermitian structure on it induced from the canonical structure on $\mathcal{A}^n\, ($Lemma $(2.2)$ of $($\cite{1}$))$.
\end{remark}

\textbf{Assumption~:} Henceforth throughout the section we assume that $\mathcal{A}$ is spectrally invariant subalgebra in a $C^*$-algebra $A$.

\begin{definition}
Let $\mathcal{E}$ be a  {\it Hermitian}, f.g.p module over $\mathcal{A}$. A compatible connection on $\,\mathcal{E}\,$ is a $\, \mathbb{C}$-linear mapping
$\,\nabla : \mathcal{E} \longrightarrow \mathcal{E} \, \otimes _\mathcal{A} \Omega _{D}^1\,$ such that,
\begin{enumerate}
\item[(a)] $\nabla (\xi a) = (\nabla\xi)a + \xi \otimes da, \, \, \, \, \, \forall\, \xi \in \mathcal{E} , a \in \mathcal{A}$;
\item[(b)] $\langle \, \xi , \nabla \eta \, \rangle - \langle \, \nabla \xi , \eta \, \rangle = d\langle \, \xi , \eta \, \rangle_\mathcal{A}\, \, \, \, \, \, \, \forall\, \xi , \eta \in \mathcal{E}\, \, \, \,$(Compatibility).
\end{enumerate}
\end{definition}

The meaning of the last equality in $\Omega_D^1$ is, if $ \nabla(\xi ) = \sum\xi_j\otimes \omega_j $, with $\xi_j \in \mathcal{E}\, , \, \omega_j \in \Omega_D^1(\mathcal{A})$, then
$ \langle \nabla \xi, \eta\rangle = \sum \omega_j^* \langle\xi_j , \eta\rangle_\mathcal{A}$. Existence of compatible connection has been discussed in (\cite{4}). Take $\mathcal{E}=p\mathcal{A}^n$, $p\in M_n(\mathcal{A})$
a projection, with canonical Hermitian structure on it. The map $\nabla_0:\xi\longmapsto p(d\xi_1,\ldots,d\xi_n)$ is a compatible connection on $\mathcal{E}$. It is called the Grassmannian connection. The space of compatible
connections is an affine space over the vector space $Hom_\mathcal{A}(\mathcal{E},\mathcal{E}\otimes _\mathcal{A} \Omega_{D}^1)$ and denoted by $Con(\mathcal{E})$. The connection $\nabla$ extends to a unique linear map
$\, \nabla^\prime : \mathcal{E} \otimes \Omega_D^1 \longrightarrow \mathcal{E} \otimes \Omega_D^2\, $ such that,
\begin{eqnarray*}
\nabla^\prime (\xi \otimes \omega) = (\nabla \xi)\omega + \xi \otimes d\omega, \, \, \, \, \, \, \forall \, \xi \in \mathcal{E}, \, \, \omega \in \Omega_D^1.
\end{eqnarray*}
It can be easily checked that $\nabla^\prime$, defined above, satisfies the Leibniz rule, i,e.
\begin{eqnarray*}
\nabla^\prime(\eta a) = \nabla^\prime(\eta)a - \eta da \, , \, \, \, \, \, \forall \, a \in \mathcal{A}\, ,\eta \in \mathcal{E} \otimes \Omega_D^1\, . 
\end{eqnarray*}
A simple calculation shows that $\varTheta = \nabla^\prime \circ \nabla$ is an element of $Hom_\mathcal{A}(\mathcal{E},\mathcal{E} \otimes_\mathcal{A} \Omega_D^2)$. 

\begin{definition}
For a connection $\nabla$, $\varTheta$ is called the curvature of the connection.
\end{definition}

Throughout this section `$u$' will stand for the rank one projection $|e_0\rangle \langle e_0|=I-l^*l$ in $\mathcal{B}\left(\ell^2(\mathbb{N})\right)$. The map $\, \phi : a \longmapsto a\otimes u$ gives an algebra embedding
of $\mathcal{A}$ in $\varSigma^2 \mathcal{A}$ and hence extends to the map $$\,\widetilde{\phi}:M_q(\mathcal{A})\longrightarrow M_q(\varSigma^2 \mathcal{A})$$ $$\quad\textbf{a}=(a_{ij})\longmapsto (a_{ij}\otimes u)_{ij}$$
By dfinition of projective module, let $\mathcal{E}=p\mathcal{A}^n$ for some natural number $n$ and an idempotent $p \in M_n(\mathcal{A})\, $. For $p=(p_{ij})_{ij}$, if we denote the matrix $(p_{ij}\otimes u)_{ij}$ by $p\otimes u$, then $\widetilde{\phi}$ gives
a f.g.p right $\varSigma^2 \mathcal{A}$-module $\widetilde {\mathcal{E}} = (p\otimes u)(\varSigma^2 \mathcal{A})^n$. However, note that $\widetilde {\mathcal{E}} = (p\otimes u)(\varSigma^2 \mathcal{A})^n$ is same as
$(p\otimes u)(\mathcal{A}\otimes \mathcal{S})^n$ because $u$ is a rank one projection operator. We recall Theorem ($3.3$) from (\cite{1}).

\begin{theorem}[\cite{1}]\label{theorem of CG}
Let $\mathcal{E}$ be a f.g.p $\mathcal{A}$-module with a Hermitian structure where $\mathcal{A}$ is spectrally invariant subalgebra in a $C^*$-algebra. Then we can have a self-adjoint idempotent $p\in M_n(\mathcal{A})$ such
that $\mathcal{E} = p\mathcal{A}^n$ and $\mathcal{E}$ has the induced canonical Hermitian structure.
\end{theorem}

Goal of this section is to prove the following theorem.
\begin{theorem}\label{extension thm}
Let $\mathcal{E}$ be a f.g.p. module over $\mathcal{A}$ equipped with a Hermitian structure $\langle\,\,,\,\rangle_\mathcal{A}$. Choose a projection $p \in M_n(\mathcal{A})$ such that $\mathcal{E}= p\mathcal{A}^n$ and
$\mathcal{E}$ has the induced canonical Hermitian structure. Let $\widetilde{\mathcal{E}} = (p\otimes u)(\varSigma^2\mathcal{A})^n$ and restrict the canonical structure on $(\varSigma^2\mathcal{A})^n$ to
$\widetilde{\mathcal{E}}$. We have an one-one affine morphism $\widetilde\phi_{con}:Con(\mathcal E)\longrightarrow Con(\widetilde{\mathcal{E}})$ which preserves the Grassmannian conections, and an one-one
$\mathbb{C}$-linear map $\,\psi:Hom_\mathcal{A}\left(\mathcal{E}, \mathcal{E}\otimes_\mathcal{A}\Omega_D^2(\mathcal{A})\right)\longrightarrow Hom_{\varSigma^2\mathcal{A}}\left(\widetilde{\mathcal{E}},\widetilde{\mathcal{E}}
\otimes_{\varSigma^2\mathcal{A}}\Omega_{\varSigma^2D}^2(\varSigma^2\mathcal{A})\right)$ such that the following diagram
\begin{center}
\begin{tikzpicture}[node distance=3cm,auto]
\node (Up)[label=above:$\widetilde\phi_{con}$]{};
\node (A)[node distance=3.5cm,left of=Up]{$Con(\mathcal{E})$};
\node (B)[node distance=4cm,right of=Up]{$Con(\widetilde{\mathcal{E}})$};
\node (Down)[node distance=1.5cm,below of=Up, label=below:$\psi$]{};
\node(C)[node distance=3.5cm,left of=Down]{$Hom_\mathcal{A}\left(\mathcal{E},\mathcal{E}\otimes_\mathcal{A}\Omega_D^2(\mathcal{A})\right)$};
\node(D)[node distance=4cm,right of=Down]{$Hom_{\varSigma^2\mathcal{A}}\left(\widetilde{\mathcal{E}},\widetilde{\mathcal{E}}\otimes_{\varSigma^2\mathcal{A}}\Omega_{\varSigma^2D}^2(\varSigma^2\mathcal{A})\right)$};
\draw[->](A) to (B);
\draw[->](C) to (D);
\draw[->](B)to node{{ $f$}}(D);
\draw[->](A)to node[swap]{{ $f$}}(C);
\end{tikzpicture}
\end{center}
commutes. Here $f$ is the map which sends any compatible connection to its associated curvature.
\end{theorem}

\begin{remark}
Choice of such a projection $p\in M_n(\mathcal{A})$ of Theorem $(\,\ref{theorem of CG}\,)$, such that $\mathcal{E}=p\mathcal{A}^n$, has the advantage that now we have to deal with the canonical Hermitian structure, which is
much easier to tackle as compared to arbitrary Hermitian structure. This is one of the main reason for our assumption that $\mathcal{A}$ is spectrally invariant subalgebra in a $C^*$-algebra because this assumption is
crucial for Theorem $(\,\ref{theorem of CG}\,)$ to hold.
\end{remark}

We break the proof of this theorem into several lemmas and propositions to make it transparent and then combine them together at the end. 
\begin{lemma}\label{imp iso 0}
As right $\varSigma^2 \mathcal{A}$ module,
\begin{eqnarray*}
(p\otimes u)(\Omega_D^1\otimes \mathcal{S})^n \cong p(\Omega_D^1)^n\otimes u\mathcal{S}
\end{eqnarray*}
\end{lemma}

\begin{proof}
We define $$\Phi : p(\Omega_D^1)^n\otimes u\mathcal{S} \longrightarrow (p\otimes u)(\Omega_D^1\otimes \mathcal{S})^n$$ $$\quad\quad\quad\, p(\omega_1,\ldots,\omega_n)\otimes uT \longmapsto (p\otimes u)(\omega_1\otimes T,
\ldots,\omega_n\otimes T)$$ and $$\Psi : (p\otimes u)(\Omega_D^1\otimes \mathcal{S})^n\longrightarrow p(\Omega_D^1)^n\otimes u\mathcal{S}$$ $$\quad\quad\quad(p\otimes u)(\omega_1\otimes T_1,\ldots,\omega_n\otimes T_n)
\longmapsto \sum_{i=1}^n p(0,\ldots,\omega_i,\ldots,0)\otimes uT_i\,.$$ Proof is now routine verification.
\end{proof}

\begin{lemma}\label{imp iso}
$(p\otimes u)(\mathcal{A}\otimes \mathcal{S})^n \cong p\mathcal{A}^n\otimes u\mathcal{S}\, \, $ as right $\, \varSigma^2 \mathcal{A}$-module.
\end{lemma}

\begin{proof}
Exact similar description of $\Phi,\Psi$ of previous Lemma ($\,$\ref{imp iso 0}$\,$) gives the isomorphism.
\end{proof}
\medskip

\textbf{Notation}~: Henceforth throughout the article $\delta(T)=[N,T]$ for all $T\in \mathcal{S}$ and $\underbrace{(0,\ldots,a_i,\ldots,0)}_{n\,\,tuple}$ will denote the element of $\mathcal{A}^n$ with $a_i\in \mathcal{A}$
at the $i$-th co-ordinate and zero elsewhere.
\medskip

\begin{proposition}\label{extended connection}
Let $\, \nabla : \mathcal{E} \longrightarrow \mathcal{E}\otimes_{\mathcal{A}} \Omega_D^1(\mathcal{A})$ be a connection where $\mathcal{E} = p\mathcal{A}^n$. Define,
$$\widetilde \nabla : (p\otimes u)(\mathcal{A}\otimes \mathcal{S})^n \longrightarrow (p\otimes u)(\mathcal{A}\otimes \mathcal{S})^n \otimes_{\varSigma^2 \mathcal{A}} \Omega_{\varSigma^2 D}^1(\varSigma^2 \mathcal{A})$$
by the rule,
\begin{eqnarray*}
(p\otimes u)(a_1\otimes T_1,\ldots,a_n\otimes T_n)\, \longmapsto &  & \sum_{i=1}^n \nabla(p(0,\ldots,a_i,\ldots,0))\otimes uT_i\\
                                                               &  & \,\,\,\,+ (p\otimes u)\left(a_1\otimes \delta(uT_1),\ldots,a_n\otimes \delta(uT_n)\right)
\end{eqnarray*}
where $\delta(T)=[N,T]$. Then $\widetilde \nabla$ defines a connection on $\widetilde {\mathcal{E}}=(p\otimes u)(\varSigma^2\mathcal{A})^n$.
\end{proposition}

\begin{proof}
Well-definedness is easy to check. Now consider any $a\otimes T+f \in \varSigma^2 \mathcal{A}\,$. Then,
\begin{eqnarray*}
&   & (p\otimes u)(a_1\otimes T_1,\ldots,a_n\otimes T_n)(a\otimes T+f)\\
& = & (p\otimes u)(a_1a\otimes T_1T,\ldots,a_na\otimes T_nT) + (p\otimes u)(a_1\otimes T_1f,\ldots,a_n\otimes T_nf).
\end{eqnarray*}
Image of this element under $\widetilde\nabla$ is, 
\begin{eqnarray*}
&  & \sum_{i=1}^n \nabla(p(0,\ldots,a_ia,\ldots,0))\otimes uT_iT + (p\otimes u)(a_1a\otimes \delta(uT_1T),\ldots,a_na\otimes \delta(uT_nT))\\
&  & \,\,\,\,\,\,+\sum_{i=1}^n \nabla(p(0,\ldots,a_i,\ldots,0))\otimes uT_if + (p\otimes u)(a_1\otimes \delta(uT_1f),\ldots,a_n\otimes \delta(uT_nf))
\end{eqnarray*}
Now,
\begin{eqnarray*}
&   & \widetilde \nabla \left((p\otimes u)(a_1\otimes T_1,\ldots,a_n\otimes T_n) \right).(a\otimes T+f) + (p\otimes u)(a_1\otimes T_1,\ldots,a_n\otimes T_n)\otimes \widetilde d(a\otimes T+f)\\
& = & \sum_{i=1}^n \nabla(p(0,\ldots,a_i,\ldots,0))a\otimes uT_iT + \sum_{i=1}^n \nabla(p(0,\ldots,a_i,\ldots,0))\otimes uT_if\\
&   & + (p\otimes u)(a_1\otimes \delta(uT_1),\ldots,a_n\otimes \delta(uT_n)).(a\otimes T+f)\\
&   & + (p\otimes u)(a_1\otimes T_1,\ldots,a_n\otimes T_n)\otimes (da\otimes T + a\otimes \delta T +1\otimes \delta f)\\
& = & \sum_{i=1}^n \{\nabla(p(0,\ldots,a_i,\ldots,0))a\otimes uT_iT + (p(0,\ldots,a_i,\ldots,0)\otimes da)\otimes uT_iT\}\\
&   & + \{(p\otimes u)(a_1a\otimes \delta(uT_1)T,\ldots,a_na\otimes \delta(uT_n)T) + (p\otimes u)(a_1a\otimes T_1\delta T,\ldots,a_na\otimes T_n\delta T)\}\\
&   & + \{(p\otimes u)(a_1\otimes \delta(uT_1)f,\ldots,a_n\otimes \delta(uT_n)f) + (p\otimes u)(a_1\otimes T_1\delta f,\ldots,a_n\otimes T_n\delta f)\}\\
&   & + \sum_{i=1}^n \nabla(p(0,\ldots,a_i,\ldots,0))\otimes uT_if\\
& = & \sum_{i=1}^n \{\nabla(p(0,\ldots,a_ia,\ldots,0))\otimes uT_iT + \sum_{i=1}^n \nabla(p(0,\ldots,a_i,\ldots,0))\otimes uT_if\\
&   & + (p\otimes u)(a_1a\otimes \delta(uT_1T),\ldots,a_na\otimes \delta(uT_nT)) + (p\otimes u)(a_1\otimes \delta(uT_1f),\ldots,a_n\otimes \delta(uT_nf))
\end{eqnarray*}
This shows that,
\begin{eqnarray*}
&   & \widetilde \nabla \left((p\otimes u)(a_1\otimes T_1,\ldots,a_n\otimes T_n).(a\otimes T+f)\right)\\
& = & \widetilde \nabla \left((p\otimes u)(a_1\otimes T_1,\ldots,a_n\otimes T_n) \right).(a\otimes T+f) + (p\otimes u)(a_1\otimes T_1,\ldots,a_n\otimes T_n)\otimes \widetilde d(a\otimes T+f)
\end{eqnarray*}
i,e. $\widetilde\nabla$ is a connection on $\widetilde{\mathcal{E}}$.
\end{proof}

\begin{lemma}\label{preservation of Grassmannian}
If $\,\nabla$ is the Grassmannian connection on $\mathcal{E}$, then $\widetilde\nabla$ is the Grassmannian connection on $\widetilde{\mathcal{E}}$. 
\end{lemma}

\begin{proof}
Let $\nabla=\nabla_0(\mathcal{E})$, the Grassmannian connection on $\mathcal{E}$. Then
\begin{eqnarray*}
&   & \widetilde{\nabla_0}((p\otimes u)(a_1\otimes T_1,\ldots,a_n\otimes T_n))\\
& = & \sum_{i=1}^n \nabla_0(\mathcal{E})(p(0,\ldots,a_i,\ldots,0))\otimes uT_i + (p\otimes u)(a_1\otimes \delta(uT_1),\ldots,a_n\otimes \delta(uT_n))\\
& = & \sum_{i=1}^n \nabla_0(\mathcal{E})((p_{1i}a_i,\ldots,p_{ni}a_i))\otimes uT_i + (p\otimes u)(a_1\otimes \delta(uT_1),\ldots,a_n\otimes \delta(uT_n))\\
& = & \sum_{i=1}^n p(d(p_{1i}a_i),\ldots,d(p_{ni}a_i))\otimes uT_i + (p\otimes u)(a_1\otimes \delta(uT_1),\ldots,a_n\otimes \delta(uT_n))\,.
\end{eqnarray*}
Now if $\nabla_0(\widetilde{\mathcal{E}})$ denotes the Grassmannian connection on $\widetilde{\mathcal{E}}\,$, then
\begin{eqnarray*}
&   & \nabla_0(\widetilde{\mathcal{E}})((p\otimes u)(a_1\otimes T_1,\ldots,a_n\otimes T_n))\\
& = & \nabla_0(\widetilde{\mathcal{E}})(\sum_{j=1}^np_{1j}a_j\otimes uT_j,\ldots,\sum_{j=1}^np_{nj}a_j\otimes uT_j)\\
& = & \sum_{j=1}^n (p\otimes u)(\widetilde d(p_{1j}a_j\otimes uT_j),\ldots,\widetilde d(p_{1j}a_j\otimes uT_j))\\
& = & \sum_{j=1}^n (p\otimes u)(d(p_{1j}a_j)\otimes uT_j+p_{1j}a_j\otimes \delta(uT_j),\ldots,d(p_{nj}a_j)\otimes uT_j+p_{nj}a_j\otimes \delta(uT_j))\,.
\end{eqnarray*}
Here $\,\widetilde d:\varSigma^2\mathcal{A}\longrightarrow \Omega_{\varSigma^2D}^1(\varSigma^2\mathcal{A})$ is the differential of Part $(3)$, Proposition ($\,$\ref{final thm}$\,$). Notice that,
$$\sum_{j=1}^n (p\otimes u)(p_{1j}a_j\otimes \delta(uT_j),\ldots,p_{nj}a_j\otimes \delta(uT_j))=(p\otimes u)(a_1\otimes \delta(uT_1),\ldots,a_n\otimes \delta(uT_n))$$ and this completes the proof. 
\end{proof}

\begin{proposition}\label{extended curvature}
Let $\,\varTheta:p\mathcal{A}^n\longrightarrow p\mathcal{A}^n\otimes_{\mathcal{A}}\Omega_D^2(\mathcal{A})$ be the curvature of the connection $\nabla$ on $\mathcal{E}=p\mathcal{A}^n$ and $\,\widetilde{\varTheta}$ denotes the
curvature of the connection $\widetilde \nabla$ of Proposition $(\,\ref{extended connection}\,)$. Then
\begin{center}
$\quad\quad\widetilde \varTheta : (p\otimes u)(\mathcal{A}\otimes \mathcal{S})^n \longrightarrow (p\otimes u)(\mathcal{A}\otimes \mathcal{S})^n \otimes_{\varSigma^2 \mathcal{A}} \Omega_{\varSigma^2 D}^2(\varSigma^2 \mathcal{A})$
\end{center} is the map given by
\begin{center}
$(p\otimes u)(a_1\otimes T_1,\ldots,a_n\otimes T_n) \longmapsto \sum_{i=1}^n \varTheta(p(0,\ldots,a_i,\ldots,0))\otimes uT_i\,.$
\end{center}
\end{proposition}

\begin{proof}
Let $\nabla : p\mathcal{A}^n \longrightarrow p\mathcal{A}^n \otimes_{\mathcal{A}} \Omega_D^1$ be a connection for which $\varTheta$ is the curvature and $\widetilde \nabla$ denotes the connection in Proposition
($\,$\ref{extended connection}$\,$). We let $\widetilde \nabla^\prime$ be the extended map,
\begin{center}
$\widetilde \nabla^\prime : \widetilde{\mathcal{E}}\otimes_{\varSigma^2 \mathcal{A}} \Omega_{\varSigma^2 D}^1(\varSigma^2 \mathcal{A}) \longrightarrow \widetilde{\mathcal{E}}\otimes_{\varSigma^2 \mathcal{A}}
\Omega_{\varSigma^2 D}^2(\varSigma^2 \mathcal{A}).$
\end{center}
Then $\,\widetilde \varTheta = \widetilde \nabla^\prime \circ \widetilde \nabla$. Now,
\begin{eqnarray*}
&   & \widetilde \nabla^\prime \left(\sum_{i=1}^n \nabla(p(0,\ldots,a_i,\ldots,0))\otimes uT_i\right)\\
& = & \widetilde \nabla^\prime \left(\sum_{i=1}^n p(\omega_1^{(i)},\ldots,\omega_n^{(i)})\otimes uT_i\right)\\
& = & \widetilde \nabla^\prime \left(\sum_{i=1}^n (p\otimes u)(\omega_1^{(i)}\otimes uT_i,\ldots,\omega_n^{(i)}\otimes uT_i) \right)\\
& = & \widetilde \nabla^\prime \left(\sum_{i=1}^n \sum_{j=1}^n (p\otimes u)(0,\ldots,\underbrace{1\otimes u}_{j-th\, place},\ldots,0)\otimes (\omega_j^{(i)}\otimes uT_i) \right)\\
& = & \sum_{i=1}^n \sum_{j=1}^n \widetilde \nabla \left((p\otimes u)(0,\ldots,\underbrace{1\otimes u}_{j-th\, place},\ldots,0) \right) . (\omega_j^{(i)}\otimes uT_i)\\
&   & + (p\otimes u)(0,\ldots,\underbrace{1\otimes u}_{j-th\, place},\ldots,0)\otimes (d\omega_j^{(i)}\otimes uT_i)\\
& = & \sum_{i=1}^n \sum_{j=1}^n \nabla (p(0,\ldots,\underbrace{1}_{j-th\, place},\ldots,0))\omega_j^{(i)}\otimes uT_i + (p\otimes u)(0,\ldots,\underbrace{d\omega_j^{(i)}\otimes uT_i}_{j-th\, place},\ldots,0)
\end{eqnarray*}
In the last equality use the fact that $\delta(u)=[N,u] = 0$. Also,
\begin{eqnarray*}
&   & \widetilde \nabla^\prime ((p\otimes u)(a_1\otimes \delta(uT_1),\ldots,a_n\otimes \delta(uT_n)))\\
& = & \widetilde \nabla^\prime (\sum_{i=1}^n (p\otimes u)(0,\ldots,a_i\otimes u,\ldots,0)\otimes (1\otimes \delta(uT_i)))\\
& = & \sum_{i=1}^n \widetilde \nabla ((p\otimes u)(0,\ldots,a_i\otimes u,\ldots,0))(1\otimes \delta(uT_i))
 + (p\otimes u)(0,\ldots,a_i\otimes u,\ldots,0)\otimes d(1\otimes \delta(uT_i))\\
& = & \sum_{i=1}^n \{\nabla (p(0,\ldots,a_i,\ldots,0))\otimes u + (p\otimes u)(0,\ldots,a_i\otimes \delta(u),\ldots,0)\}(1\otimes \delta(uT_i))\\
& = & \sum_{i=1}^n \{ \nabla(p(0,\ldots,a_i,\ldots,0))\otimes u\} (1\otimes \delta(uT_i))\\
& = & 0\, .\, \, \quad \, (\,see\,\,\,Theorem\,\,(\,\ref{final thm}\,))
\end{eqnarray*}
Finally,
\begin{eqnarray*}
&   & \widetilde \varTheta \left((p\otimes u)(a_1\otimes T_1,\ldots,a_n\otimes T_n)\right)\\
& = & \widetilde \nabla^\prime \left(\sum_{i=1}^n \nabla(p(0,\ldots,a_i,\ldots,0))\otimes uT_i\right)
 + \widetilde \nabla^\prime \left((p\otimes u)(a_1\otimes \delta(uT_1),\ldots,a_n\otimes \delta(uT_n))\right)\\
& = & \sum_{i=1}^n \sum_{j=1}^n \nabla (p(0,\ldots,\underbrace{1}_{j-th\, place},\ldots,0))\omega_j^{(i)}\otimes uT_i
 + (p\otimes u)(0,\ldots,\underbrace{d\omega_j^{(i)}\otimes uT_i}_{j-th\, place},\ldots,0)\\
& = & \sum_{i=1}^n \sum_{j=1}^n (\nabla (p(0,\ldots,\underbrace{1}_{j-th\, place},\ldots,0))\omega_j^{(i)} + p(0,\ldots,\underbrace{1}_{j-th\, place},\ldots,0)\otimes d\omega_j^{(i)})\otimes uT_i\\
& = & \sum_{i=1}^n \sum_{j=1}^n \nabla^\prime (p(0,\ldots,\underbrace{1}_{j-th\, place},\ldots,0)\otimes \omega_j^{(i)})\otimes uT_i\\
& = & \sum_{i=1}^n \sum_{j=1}^n \nabla^\prime(p(0,\ldots,\underbrace{\omega_j^{(i)}}_{j-th\, place},\ldots,0))\otimes uT_i\\
& = & \sum_{i=1}^n \varTheta(p(0,\ldots,a_i,\ldots,0))\otimes uT_i
\end{eqnarray*}
\end{proof}

\begin{lemma}
Let $\,\widetilde\xi,\,\widetilde\eta\in(p\otimes u)(\mathcal{A}\otimes \mathcal{S})^n$ and $\Psi(\widetilde \xi) = \sum_k\xi_k\otimes uT_k\, $, $\Psi(\widetilde\eta) = \sum_k\eta_k\otimes uS_k\,$ where
\begin{center}
$\,\,\Psi : (p\otimes u)(\mathcal{A}\otimes \mathcal{S})^n \longrightarrow p\mathcal{A}^n\otimes u\mathcal{S}$
\end{center}
\begin{center}
$\quad\quad\quad\quad(p\otimes u)(a_1\otimes T_1,\ldots,a_n\otimes T_n)\longmapsto\displaystyle\sum_{i=1}^n\,p(0,\ldots,a_i,\ldots,0)\otimes uT_i$
\end{center}
is the isomorphism of Lemma $(\,\ref{imp iso}\,)$. Then the induced canonical Hermitian structure on $(p\otimes u)(\mathcal{A}\otimes \mathcal{S})^n$ has the following form
\begin{eqnarray*}
\langle\, \widetilde \xi,\widetilde \eta\, \rangle_{\varSigma^2 \mathcal{A}} = \sum_{k,k^\prime} \langle\, \xi_k,\eta_{k^\prime} \rangle_{\mathcal{A}}\otimes (uT_k)^*(uS_{k^\prime})\, .
\end{eqnarray*} 
\end{lemma}

\begin{proof}
Let $\,\widetilde \xi = (p\otimes u)(a_1\otimes T_1,\ldots,a_n\otimes T_n)$ and $\widetilde \eta = (p\otimes u)(a_1^\prime \otimes T_1^\prime,\ldots,a_n^\prime\otimes T_n^\prime)$. Then
\begin{eqnarray*}
 \langle\, \widetilde \xi,\widetilde \eta \rangle_{\varSigma^2 \mathcal{A}}
& = & \langle\, \, (\sum_{j=1}^n p_{1j}a_j\otimes uT_j,\ldots,\sum_{j=1}^n p_{nj}a_j\otimes uT_j)\, ,
(\sum_{j=1}^n p_{1j}a_j^\prime \otimes uT_j^\prime,\ldots,\sum_{j=1}^n p_{nj}a_j^\prime \otimes uT_j^\prime)\, \rangle\\
& = & \langle \, \, \sum_{j=1}^n (p_{1j}a_j\otimes uT_j,\ldots,p_{nj}a_j\otimes uT_j)\, ,\sum_{j=1}^n (p_{1j}a_j^\prime \otimes uT_j^\prime,\ldots,p_{nj}a_j^\prime \otimes uT_j^\prime)\, \rangle\\
& = & \sum_{j,l=1}^n \langle \, \, (p_{1j}a_j\otimes uT_j,\ldots,p_{nj}a_j\otimes uT_j)\, ,\, (p_{1l}a_l^\prime \otimes uT_l^\prime,\ldots,p_{nl}a_l^\prime \otimes uT_l^\prime)\, \rangle\\
& = & \sum_{j,l=1}^n \sum_{k=1}^n (p_{kj}a_j\otimes uT_j)^*(p_{kl}a_l^\prime \otimes uT_l^\prime)
\end{eqnarray*}
Now $\Psi(\widetilde \xi) = \sum_{j=1}^n p(0,\ldots,a_j,\ldots,0)\otimes uT_j$ and $\Psi(\widetilde \eta) = \sum_{j=1}^n p(0,\ldots,a_j^\prime,\ldots,0)\otimes uT_j^\prime\,$. Let $\xi_j = p(0,\ldots,a_j,\ldots,0)\, $ and 
$\, \eta_j = p(0,\ldots,a_j^\prime,\ldots,0)$. Now,
\begin{eqnarray*}
\langle\, \xi_r,\eta_s \rangle_{\mathcal{A}} & = & \langle\, p(0,\ldots,a_r,\ldots,0),p(0,\ldots,a_s^\prime,\ldots,0)\rangle_{\mathcal{A}}\\
                                             & = & \langle\, (p_{1r}a_r,\ldots,p_{nr}a_r),(p_{1s}a_s^\prime,\ldots,p_{ns}a_s^\prime)\rangle_{\mathcal{A}}\\
                                             & = & \sum_{k=1}^n (p_{kr}a_r)^*p_{ks}a_s^\prime\\
                                             & = & \sum_{k=1}^n a_r^*p_{kr}p_{ks}a_s^\prime
\end{eqnarray*}
Hence we have
\begin{center}
$\sum_{r,s} \langle\, \xi_r,\eta_s\rangle_{\mathcal{A}}\otimes (uT_r)^*(uT^\prime_s) = \sum_{r,s} \left(\sum_{i=1}^n (p_{ir}a_r\otimes uT_r)^*(p_{is}a_s^\prime \otimes uT^\prime_s)\right)\,.$
\end{center}
\end{proof}

\begin{lemma}\label{useful in next lemma}
For $\xi,\eta \in \mathcal{E}$ we have
\begin{enumerate}
 \item[(a)] $\langle \xi\otimes uT,\nabla \eta\otimes uS\rangle = \langle \xi,\nabla\eta\rangle\otimes (uT)^*uS\, $,
 \item[(b)] $\langle \nabla \xi\otimes uT,\eta\otimes uS \rangle = \langle \nabla \xi,\eta \rangle\otimes (uT)^*uS\, $.
\end{enumerate}
\end{lemma}

\begin{proof}
Let $\xi = p(a_1,\ldots,a_n) \in p\mathcal{A}^n$, $\nabla \eta = \sum_i p(b_{1i},\ldots,b_{ni})\otimes \omega_i \in p\mathcal{A}^n\otimes \Omega_D^1(\mathcal{A})$. Then,
\begin{eqnarray*}
\xi\otimes uT & = & p(a_1,\ldots,a_n)\otimes uT\\
              & = & (p\otimes u)(a_1\otimes T,\ldots,a_n\otimes T)
\end{eqnarray*}
and
\begin{eqnarray*}
\nabla \eta\otimes uS & = & \sum_i (p(b_{1i},\ldots,b_{ni})\otimes \omega_i)\otimes uS\\
                      & = & \sum_i p(b_{1i}\omega_i,\ldots,b_{ni}\omega_i)\otimes uS\\
                      & = & \sum_i (p\otimes u)(b_{1i}\omega_i\otimes S,\ldots,b_{ni}\omega_i\otimes S)\\
                      & = & \sum_i (p\otimes u)(b_{1i}\otimes u,\ldots,b_{ni}\otimes u)\otimes (\omega_i\otimes uS)
\end{eqnarray*}
Hence,
\begin{eqnarray*}
&   & \langle \xi\otimes uT,\nabla \eta\otimes uS \rangle\\
& = & \sum_i \langle (p\otimes u)(a_1\otimes T,\ldots,a_n\otimes T),(p\otimes u)(b_{1i}\otimes u,\ldots,b_{ni}\otimes u)\rangle_{\varSigma^2 \mathcal{A}} (\omega_i\otimes uS)\\
& = & \sum_i \langle p(a_1,\ldots,a_n)\otimes uT,p(b_{1i},\ldots,b_{ni})\otimes u\rangle_{\varSigma^2 \mathcal{A}} (\omega_i\otimes uS)\\
& = & \sum_i \left(\langle p(a_1,\ldots,a_n),p(b_{1i},\ldots,b_{ni})\rangle_\mathcal{A}\otimes (uT)^*u\right) (\omega_i\otimes uS)\\
& = & \sum_i (\langle p(a_1,\ldots,a_n),p(b_{1i},\ldots,b_{ni})\rangle_\mathcal{A} \omega_i)\otimes (uT)^*uS\\
& = & \sum_i \langle p(a_1,\ldots,a_n),p(b_{1i},\ldots,b_{ni})\otimes \omega_i \rangle \otimes (uT)^*uS\\
& = & \langle \xi,\nabla \eta \rangle\otimes (uT)^*uS
\end{eqnarray*}
This proves part $(a)$ and part $(b)$ follows similarly.
\end{proof}

\begin{lemma}\label{extended compatibility}
The connection $\widetilde \nabla$ of Proposition $(\,\ref{extended connection}\,)$ is compatible with the Hermitian structure $\, \langle \, \, \, \, ,\, \, \, \rangle_{\varSigma^2 \mathcal{A}}$ on $\widetilde {\mathcal{E}}\, $,
if $\,\nabla$ is so with respect to $\, \langle \, \, \, \, ,\, \, \, \rangle_{\mathcal{A}}$ on $\mathcal{E}$. 
\end{lemma}

\begin{proof}
For $\, \widetilde \xi,\, \widetilde \eta \in (p\otimes u)(\mathcal{A}\otimes \mathcal{S})^n$, we have isomorphic elements $\sum \xi\otimes uT,\sum \eta\otimes uS \in p\mathcal{A}^n\otimes u\mathcal{S}\, $ respectively. Let 
$\xi = p(a_1,\ldots,a_n)$ and $\eta = p(b_1,\ldots,b_n)$. It is easy to see that,
\begin{center}
$\widetilde \nabla(\widetilde \xi) = \nabla(\xi)\otimes uT + (p\otimes u)(a_1\otimes \delta(uT),\ldots,a_n\otimes \delta(uT))$.
\end{center}
\begin{center}
$\widetilde \nabla(\widetilde \eta) = \nabla(\eta)\otimes uS + (p\otimes u)(b_1\otimes \delta(uS),\ldots,b_n\otimes \delta(uS))$.
\end{center}
Now,
\begin{eqnarray*}
&   & \langle\, \widetilde \xi,\widetilde \nabla \widetilde \eta \rangle_{\varSigma^2 \mathcal{A}} - \langle\, \widetilde \nabla \widetilde \xi,\widetilde \eta \rangle_{\varSigma^2 \mathcal{A}}\\
& = & \langle\, \xi \otimes uT,\nabla \eta \otimes uS \rangle_{\varSigma^2 \mathcal{A}} - \langle\, \nabla \xi \otimes uT,\eta \otimes uS \rangle_{\varSigma^2 \mathcal{A}}\\
&   & + \langle\, \xi \otimes uT,(p\otimes u)(b_1\otimes \delta(uS),\ldots,b_n\otimes \delta(uS)) \rangle_{\varSigma^2 \mathcal{A}}\\
&   & - \langle\, (p\otimes u)(a_1\otimes \delta(uT),\ldots,a_n\otimes \delta(uT),\eta\otimes uS \rangle_{\varSigma^2 \mathcal{A}}\\
& = & (\langle\, \xi,\nabla \eta\rangle_\mathcal{A} - \langle\, \nabla\xi,\eta\rangle_\mathcal{A})\otimes (uT)^*(uS)\\
&   & + \langle\, \xi\otimes uT,\eta\otimes u\delta(uS)\rangle_{\varSigma^2 \mathcal{A}} - \langle\, \xi\otimes u\delta(uT),\eta\otimes uS\rangle_{\varSigma^2 \mathcal{A}}
\, \, \, \, \, \, \, \, \, \, \, \, \, \, (by\quad Lemma\,\,(\,\ref{useful in next lemma}\,))\\
& = & d(\langle\, \xi,\eta \rangle_\mathcal{A})\otimes (uT)^*(uS) + \langle\, \xi,\eta \rangle_\mathcal{A}\otimes (uT)^*u\delta(uS) - \langle\, \xi,\eta \rangle_\mathcal{A}\otimes (u\delta(uT))^*uS
\end{eqnarray*}
Finally,
\begin{eqnarray*}
\widetilde d(\langle\, \widetilde \xi,\widetilde \eta\rangle_{\varSigma^2 \mathcal{A}}) & = & \widetilde d(\langle\, \xi\otimes uT,\eta \otimes uS \rangle_{\varSigma^2 \mathcal{A}}\\
& = & \widetilde d(\langle\, \xi,\eta \rangle_\mathcal{A}\otimes (uT)^*uS)\\
& = & d(\langle\, \xi,\eta \rangle_\mathcal{A})\otimes (uT)^*uS + \langle\, \xi,\eta \rangle_\mathcal{A}\otimes \delta((uT)^*uS)\\
& = & d(\langle\, \xi,\eta \rangle_\mathcal{A})\otimes (uT)^*uS + \langle\, \xi,\eta \rangle_\mathcal{A}\otimes ((uT)^*\delta(uS) + \delta((uT)^*)uS)\\
& = & d(\langle\, \xi,\eta \rangle_\mathcal{A})\otimes (uT)^*uS + \langle\, \xi,\eta \rangle_\mathcal{A}\otimes ((uT)^*\delta(uS) - (\delta(uT))^*uS)\\
\end{eqnarray*}
This shows compatibility of $\widetilde \nabla$.
\end{proof}

\textbf{Proof of the Theorem $\,$\ref{extension thm}$\,$:} Define $$\,\widetilde\phi_{con}:Con(\mathcal{E})\,\longrightarrow Con(\widetilde{\mathcal{E}})$$ $$\quad\quad\quad\nabla\longmapsto\widetilde\nabla$$ where
$\widetilde\nabla$ is as defined in Proposition ($\,$\ref{extended connection}$\,$). Lemma ($\,$\ref{extended compatibility}$\,$) proves that $\mathcal{R}an(\widetilde\phi_{con})\subseteq Con(\widetilde{\mathcal{E}})$ and
Lemma ($\,$\ref{preservation of Grassmannian}$\,$) proves preservation of the Grassmannian connection. It is easy to check that $\widetilde\phi_{con}$ is an affine morphism between $Con(\mathcal E)$ and $Con(\widetilde{
\mathcal E})$. To see injectivity, let $\widetilde\phi_{con}(\nabla_1)=\widetilde\phi_{con}(\nabla_2)$ and choose any $\xi=p(a_1,\ldots,a_n)\in \mathcal{E}$. Then $\widetilde\xi=(p\otimes u)(a_1\otimes u,\ldots,a_n\otimes u)
\in \widetilde{\mathcal{E}}$. Then it follows that $\nabla_1(\xi)\otimes u=\nabla_2(\xi)\otimes u\, \, $(use Lemma ($\,$\ref{imp iso}$\,$)) i,e. $\, \nabla_1(\xi)=\nabla_2(\xi)$. Now define $$\quad\quad\quad\quad\quad\psi:
Hom_\mathcal{A}\left(\mathcal{E},\mathcal{E}\otimes_\mathcal{A}\Omega_D^2(\mathcal{A})\right)\longrightarrow Hom_{\varSigma^2\mathcal{A}}\left(\widetilde{\mathcal{E}},\widetilde{\mathcal{E}}\otimes_{\varSigma^2\mathcal{A}}
\Omega_{\varSigma^2D}^2(\varSigma^2\mathcal{A})\right)$$  $$\psi(g)\left((p\otimes u)(a_1\otimes T_1,\ldots,a_n\otimes T_n)\right):=\sum_{i=1}^ng\left(p(0,\ldots,a_i,\ldots,0)\right)\otimes uT_i\,.$$ It is easy to see that
$\psi$ is a well-defined linear map because $\widetilde{\mathcal{E}}\otimes_{\varSigma^2\mathcal{A}}\Omega_{\varSigma^2D}^2(\varSigma^2\mathcal{A})\cong p\left(\Omega_D^2(\mathcal{A})\right)^n\otimes u\mathcal{S}$ as right
$\varSigma^2\mathcal{A}\, $-module~ (proof of this fact goes on the same route as described in Lemma ($\,$\ref{imp iso 0}$\,$)). Injectivity follows similarly as before. Finally, in view of Proposition
($\,$\ref{extended curvature}$\,$), we see that the diagram commutes and this completes the proof.$\quad\Square$

\end{document}